\documentclass[a4paper,reqno,11pt]{amsart}

\usepackage{amsmath, amsfonts, amssymb, amsthm, amscd}
\usepackage{graphicx}
\usepackage{psfrag}
\usepackage{perpage}
\usepackage{url}
\usepackage{color}
\usepackage{mathrsfs}

\usepackage{dsfont} 

\usepackage[utf8]{inputenc}
\usepackage[T1]{fontenc}
\usepackage{microtype}

\usepackage[a4paper,scale={0.72,0.74},marginratio={1:1},footskip=7mm,headsep=10mm]{geometry}

\usepackage{hyperref}

\makeatletter
\def\@secnumfont{\bfseries\scshape}

\def\section{\@startsection{section}{1}%
  \z@{.7\linespacing\@plus\linespacing}{.5\linespacing}%
  {\normalfont\large\bfseries\scshape\centering}}

\def\subsection{\@startsection{subsection}{2}%
  \z@{.5\linespacing\@plus.7\linespacing}{-.5em}%
  {\normalfont\bfseries\scshape}}

\def\subsubsection{\@startsection{subsubsection}{3}%
  \z@{.5\linespacing\@plus.7\linespacing}{-.5em}%
  {\normalfont\scshape}}

\def\specialsection{\@startsection{section}{1}%
  \z@{\linespacing\@plus\linespacing}{.5\linespacing}%
  {\normalfont\centering\large\bfseries\scshape}}
\makeatother

\makeatletter

\renewenvironment{proof}[1][\proofname]{\par
\pushQED{\qed}%
\normalfont \topsep4\p@\@plus4\p@\relax
\trivlist
\item[\hskip\labelsep
\bfseries
#1\@addpunct{.}]\ignorespaces
}{%
\popQED\endtrivlist\@endpefalse
}
\makeatother

\makeatletter
\newcommand \Dotfill {\leavevmode \leaders \hb@xt@ 6pt{\hss .\hss }\hfill \kern \z@}
\makeatother

\makeatletter
\def\@tocline#1#2#3#4#5#6#7{\relax
  \ifnum #1>\c@tocdepth 
  \else
    \par \addpenalty\@secpenalty\addvspace{#2}%
    \begingroup \hyphenpenalty\@M
    \@ifempty{#4}{%
      \@tempdima\csname r@tocindent\number#1\endcsname\relax
    }{%
      \@tempdima#4\relax
    }%
    \parindent\z@ \leftskip#3\relax \advance\leftskip\@tempdima\relax
    \rightskip\@pnumwidth plus4em \parfillskip-\@pnumwidth
    #5\leavevmode\hskip-\@tempdima
      \ifcase #1
       \or\or \hskip 1.65em \or \hskip 3.3em \else \hskip 4.95em \fi%
      #6\nobreak\relax
    \Dotfill
    \hbox to\@pnumwidth{\@tocpagenum{#7}}\par
    \nobreak
    \endgroup
  \fi}
\makeatother

\makeatletter
\def\l@section{\@tocline{1}{0pt}{1pc}{}{\scshape}}
\renewcommand{\tocsection}[3]{%
\indentlabel{\@ifnotempty{#2}{\ignorespaces#1 #2.\hskip 0.7em}}#3}
\def\l@subsection{\@tocline{2}{0pt}{1pc}{5pc}{}}

\def\l@subsubsection{\@tocline{3}{0pt}{1pc}{7pc}{}}

\makeatother

\numberwithin{equation}{section}

\newtheoremstyle{mytheorem}{.7\linespacing\@plus.3\linespacing}{.7\linespacing\@plus.3\linespacing}%
     {\itshape}
     {}
     {\bfseries}
     {. }
     {0.3ex}
     {\thmname{{\bfseries #1}}\thmnumber{ {\bfseries #2}}\thmnote{ (#3)}}  

\theoremstyle{mytheorem}

\newtheorem{theorem}{Theorem}[section]
\newtheorem{lemma}[theorem]{Lemma}
\newtheorem{proposition}[theorem]{Proposition}

\newtheorem{remark}[theorem]{Remark}

\newcommand{\bbE}{{\ensuremath{\mathbb E}} }

\newcommand{\bbN}{{\ensuremath{\mathbb N}} }

\newcommand{\bbP}{{\ensuremath{\mathbb P}} }

\newcommand{\cB}{{\ensuremath{\mathcal B}} }
\newcommand{\cC}{{\ensuremath{\mathcal C}} }
\newcommand{\cD}{{\ensuremath{\mathcal D}} }

\newcommand{\cF}{{\ensuremath{\mathcal F}} }

\newcommand{\cI}{{\ensuremath{\mathcal I}} }

\newcommand{\cL}{{\ensuremath{\mathcal L}} }
\newcommand{\cM}{{\ensuremath{\mathcal M}} }

\newcommand{\cR}{{\ensuremath{\mathcal R}} }
\newcommand{\cS}{{\ensuremath{\mathcal S}} }
\newcommand{\cT}{{\ensuremath{\mathcal T}} }

\newcommand{\ga}{\alpha}
\newcommand{\gb}{\beta}

\newcommand{\go}{\omega}

\renewcommand{\tilde}{\widetilde}          
\DeclareMathSymbol{\leqslant}{\mathalpha}{AMSa}{"36} 
\DeclareMathSymbol{\geqslant}{\mathalpha}{AMSa}{"3E} 
\DeclareMathSymbol{\eset}{\mathalpha}{AMSb}{"3F}     
\newcommand{\dd}{\text{\rm d}}             

\newcommand{\R}{\mathbb{R}}

\newcommand{\N}{\mathbb{N}}

\def\bs{\boldsymbol}

\newcommand{\PEfont}{\mathrm}

\newcommand{\p}{\ensuremath{\PEfont P}}
\newcommand{\e}{\ensuremath{\PEfont E}}
\newcommand{\E}{\e}
\renewcommand{\P}{\p}

\newcommand\bP{\ensuremath{\bs{\mathrm{P}}}}
\newcommand\bE{\ensuremath{\bs{\mathrm{E}}}}

\DeclareMathOperator{\bbvar}{\ensuremath{\mathbb{V}ar}}

\newcommand{\ind}{\mathds{1}}

\renewcommand{\epsilon}{\varepsilon}
\renewcommand{\theta}{\vartheta}
\renewcommand{\rho}{\varrho}

\newenvironment{myenumerate}{%
\renewcommand{\theenumi}{\arabic{enumi}}%
\renewcommand{\labelenumi}{{\rm(\theenumi)}}%
\begin{list}{\labelenumi}
	{%
	\setlength{\itemsep}{0.4em}%
	\setlength{\topsep}{0.5em}%
	\setlength\leftmargin{2.45em}%
	\setlength\labelwidth{2.05em}%
	\setlength{\labelsep}{0.4em}%
	\usecounter{enumi}%
	}%
	}%
{\end{list}
}

{\end{list}
}

\newenvironment{ienumerate}{%
\renewcommand{\theenumi}{\roman{enumi}}%
\renewcommand{\labelenumi}{{\rm(\theenumi)}}%
\begin{list}{\labelenumi}
	{%
	\setlength{\itemsep}{0.4em}%
	\setlength{\topsep}{0.5em}%
	\setlength\leftmargin{2.45em}%
	\setlength\labelwidth{2.05em}%
	\setlength{\labelsep}{0.4em}%
	\usecounter{enumi}%
	}%
	}%
{\end{list}
}

\renewenvironment{enumerate}{
\begin{myenumerate}}%
{\end{myenumerate}}

\newenvironment{myitemize}{%
\begin{list}{$\bullet$}%
 	{%
	\setlength{\itemsep}{0.4em}%
	\setlength{\topsep}{0.5em}%
	\setlength\leftmargin{2.45em}%
	\setlength\labelwidth{2.05em}%
	\setlength{\labelsep}{0.4em}%
	}%
	}%
{\end{list}}

\renewenvironment{itemize}{
\begin{myitemize}}%
{\end{myitemize}}

\MakePerPage[2]{footnote} 

\def\hbeta{{\hat{\beta}}}
\def\hh{{\hat{h}}}

\def\bR{\overline{\R}}

\def\td{\mathtt d}
\def\tg{\mathtt g}

\def\dd{\mathrm{d}}

\def\hbeta{{\hat{\beta}}}
\def\hh{{\hat{h}}}

\newcommand{\asto}[1]{\underset{{#1}\to\infty}{\longrightarrow}}

\newcommand\bpsi{\boldsymbol\psi}

\newcommand\btau{\boldsymbol\tau}
\newcommand\bZ{\boldsymbol Z}

\newcommand\bX{\boldsymbol X}

\newcommand\lntwo{\lfloor n/2 \rfloor}
\newcommand\untwo{\lceil n/2 \rceil}
\newcommand\lnfour{\lfloor n/4 \rfloor}

\newcommand\bK{\overline K}

\begin{document}
\baselineskip=14pt

\title{The continuum disordered pinning model}

\begin{abstract}
Any renewal processes on $\N_0$ with
a polynomial tail, with  exponent $\alpha \in (0,1)$, has a non-trivial scaling limit,
known as the \emph{$\alpha$-stable regenerative set}.
In this paper we consider Gibbs transformations of such renewal processes
in an i.i.d.\ random environment, called disordered pinning models. We show that
for $\alpha \in (\frac{1}{2}, 1)$
these models have a universal scaling limit,
which we call the {\em continuum disordered pinning model} (CDPM).
This is a random closed subset of $\R$ in a white noise random environment,
with subtle features:
\begin{itemize}
\item Any \emph{fixed} a.s.\ property of the $\alpha$-stable regenerative set
(e.g., its Hausdorff dimension)
is also an a.s.\ property of the CDPM,
for almost every realization of the environment.
\item Nonetheless, the law of the CDPM is
\emph{singular} with respect to the law of the $\alpha$-stable regenerative set,
for almost every realization of the environment.
\end{itemize}
The existence of a disordered continuum model,
such as the CDPM, is a manifestation of \emph{disorder relevance}
for pinning models with $\alpha \in (\frac{1}{2}, 1)$.
\end{abstract}

\author[F.Caravenna]{Francesco Caravenna}
\address{Dipartimento di Matematica e Applicazioni\\
 Universit\`a degli Studi di Milano-Bicocca\\
 via Cozzi 55, 20125 Milano, Italy}
\email{francesco.caravenna@unimib.it}

\author[R.Sun]{Rongfeng Sun}
\address{
Department of Mathematics\\
National University of Singapore\\
10 Lower Kent Ridge Road, 119076 Singapore
}
\email{matsr@nus.edu.sg}

\author[N.Zygouras]{Nikos Zygouras}
\address{Department of Statistics\\
University of Warwick\\
Coventry CV4 7AL, UK}
\email{N.Zygouras@warwick.ac.uk}
\date{\today}


\maketitle

\section{Introduction}

We consider disordered pinning models,
which are defined via a Gibbs change of measure of a renewal process,
depending on an external i.i.d.\ random environment.
First introduced in the physics and biology literature,
these models have attracted much attention due to their rich structure,
which is amenable to a rigorous investigation; see, e.g.,
the monographs of Giacomin~\cite{G07,G10} and den Hollander~\cite{dH09}.

In this paper we define a \emph{continuum disordered pinning model} (CDPM),
inspired by recent work of Alberts, Khanin and Quastel~\cite{AKQ14b}
on the directed polymer in random environment.
The interest for such a continuum model is manifold:
\begin{itemize}
\item It is a \emph{universal object}, arising as the scaling limit of
discrete disordered pinning models in a suitable continuum and weak disorder limit,
cf.\ Theorem~\ref{th:main1}.

\item It provides a tool to capture the emerging effect of disorder
in pinning models, when disorder
is relevant, cf.\ Subsection~\ref{S:disrel} for a more detailed discussion.

\item It can be interpreted as an $\alpha$-stable regenerative set
in a white noise random environment, displaying subtle properties,
cf.\ Theorems~\ref{T:averageabs}, \ref{T:quenchsing} and~\ref{T:cconv3}.
\end{itemize}

Throughout the paper, we use the conventions
$\N := \{1,2,\ldots\}$, $\bbN_0:=\{0\}\cup\N$,
and write $a_n \sim b_n$ to mean $\lim_{n\to\infty} a_n/b_n = 1$.

\subsection{Renewal processes and regenerative sets}

Let $\tau:= (\tau_n)_{n\geq 0}$
be a \emph{renewal process} on $\N_0$,
that is $\tau_0 = 0$ and
the increments $(\tau_n-\tau_{n-1})_{n\in\N}$ are i.i.d. $\N$-valued random
variables (so that $0 = \tau_0 < \tau_1 < \tau_2 < \ldots$).
Probability and expectation for $\tau$ will be denoted respectively by $\p$ and $\E$.
We assume that $\tau$ is \emph{non-terminating}, i.e., $\p(\tau_1 < \infty) = 1$, and
\begin{equation} \label{eq:ass}
	K(n):=\p(\tau_1=n) = \frac{L(n)}{n^{1+\alpha}}
	,  \qquad \text{as } n\to\infty \,,
\end{equation}
where $\alpha\in (0,1)$ and $L(\cdot)$ is a slowly varying
function~\cite{BGT87}. We assume for simplicity
that $K(n)>0$ for every $n\in \bbN$
(periodicity complicates notation, but can be easily incorporated).

Let us denote by $\cC$ the space of all \emph{closed subsets of $\R$}.
There is a natural topology on $\cC$, called the Fell-Matheron
topology~\cite{F62,M75,M05}, which turns $\cC$ into a \emph{compact Polish space}, i.e.\
a compact separable topological space which admits a complete metric.
This can be taken as a version of the \emph{Hausdorff distance}
(see Appendix~\ref{A:RCS} for more details).

Identifying the renewal process $\tau=\{\tau_n\}_{n\geq 0}$ with its range,
we may view $\tau$ as a random subset of $\N_0$,
i.e.\ as a $\cC$-value random variable
(hence we write $\{n \in \tau\} := \bigcup_{k\ge 0} \{\tau_k = n\}$).
This viewpoint is very fruitful, because as $N\to\infty$
the rescaled set
\begin{equation} \label{eq:rescN}
	\frac{\tau }{N} = \bigg\{ \frac{\tau_n}{N} \bigg\}_{n\ge 0}
\end{equation}
converges in distribution on $\cC$
to a universal random closed set $\btau$ of $[0,\infty)$, called
the \emph{$\alpha$-stable regenerative set} (cf.\ \cite{FFM85},
\cite[Thm.~A.8]{G07}).
This coincides with
the closure of the range of the $\alpha$-stable subordinator
or, equivalently, with the zero level set of a Bessel process
of dimension $\delta = 2(1-\alpha)$
(see Appendix~\ref{A:RCS}),
and we denote its law by $\bP^\alpha$.

\begin{remark}\rm
Random sets have been studied extensively
\cite{M75,M05}. Here we focus on the special case of random closed subsets of $\R$.
The theory developed in \cite{FFM85}
for \emph{regenerative sets} cannot be applied
in our context, because we modify renewal processes
through inhomogeneous perturbations
and  conditioning (see \eqref{eq:Gweight}-\eqref{eq:condpm} below).
For this reason, in Appendix~\ref{A:RCS} we review and develop
a general framework to study convergence of random closed
sets of $\R$, based on a natural notion of finite-dimensional distributions.
\end{remark}

\subsection{Disordered pinning models}

Let $\omega:=(\omega_n)_{n\in\N}$ be i.i.d.\ random variables
(independent of the renewal process $\tau$), which represent the disorder.
Probability and expectation for $\omega$ will be denoted respectively by $\bbP$ and $\bbE$.
We assume that
\begin{equation}\label{eq:omegacond}
	\bbE[\omega_n]=0, \qquad
	\bbvar (\omega_n)=1, \qquad
	\exists t_0 > 0: \ \ \Lambda(t):=\log \bbE[e^{t \omega_n} ]<\infty
	\ \ \text{for $|t| \le t_0$} \,.
\end{equation}

The \emph{disordered pinning model}
is a random probability law $\P^{\omega}_{N,\beta,h}$ on subsets of $\{0, \ldots, N\}$,
indexed by realizations $\omega$ of the disorder,
the system size $N\in\N$, the disorder strength $\beta>0$
and bias $h\in\R$, defined by the following Gibbs change of measure
of the renewal process $\tau$:
\begin{equation}\label{eq:Gweight}
	\frac{\p^{\omega}_{N,\beta,h} (\tau\cap [0,N])}{\p (\tau\cap [0,N])}
	 :=\frac{1}{Z^{\omega}_{N,\beta,h}} e^{\sum_{n=1}^N (\beta \omega_n
	-\Lambda(\beta) + h) \ind_{\{n\in\tau\}}} \,,
\end{equation}
where the normalizing constant
\begin{equation}\label{eq:Zpin}
	Z_{N,\gb,h}^{\go}:=\E\Big[ e^{\sum_{n=1}^N (\beta \omega_n
	-\Lambda(\beta) + h) \ind_{\{n\in\tau\}}} \Big]
\end{equation}
is called the {\em partition function}.
In words, we perturb the law of the renewal process $\tau$ in the interval
$[0,N]$, by giving rewards/penalties $(\beta\omega_n - \Lambda(\beta) + h)$
to each visited site $n \in \tau$.
(The presence of the factor $\Lambda(\beta)$ in \eqref{eq:Gweight}-\eqref{eq:Zpin},
which just corresponds to
a translation of $h$, allows to have normalized weights
$\bbE[e^{\beta \omega_n - \Lambda(\beta)}] = 1$ for $h=0$.)

The properties of the model $\P^{\omega}_{N,\beta,h}$,
especially in the limit $N\to\infty$, have been studied
in depth in the recent mathematical literature (see e.g.~\cite{G07,G10, dH09} for an
overview).
In this paper we focus on the problem of defining a continuum analogue of $\P^{\omega}_{N,\beta,h}$.

Since under the ``free law'' $\p$
the rescaled renewal process $\tau / N$ converges in distribution to the
$\alpha$-stable regenerative set $\btau$,
it is natural to ask what happens under the ``interacting law'' $\P^{\omega}_{N,\beta,h}$.
Heuristically, in the scaling limit the i.i.d.\ random variables $(\omega_n)_{n\in\N}$
should be replaced by a one-dimensional white noise
$(\dd W_t)_{t \in [0,\infty)}$,
where $W = (W_t)_{t\in [0,\infty)}$ denotes a standard Brownian
motion (independent of $\btau$). Looking at \eqref{eq:Gweight},
a natural candidate for the scaling limit of $\tau /N$ under
$\P^{\omega}_{N,\beta,h}$ would be the random measure
$\bP^{\alpha;W}_{T,\beta, h}$ on $\cC$ defined by
\begin{equation}\label{eq:nosense}
	\frac{\dd\bP^{\alpha;W}_{T, \beta, h}}{\dd\bP^\alpha}(\btau \cap [0,T]) :=
	\frac{1}{\bZ^{\alpha; W}_{T, \beta, h}}
	e^{\int_0^T \ind_{\{t\in\btau \}} (\beta \dd W_t +  (h-\frac{1}{2}\beta^2) \dd t)} \,,
\end{equation}
where the continuum partition function
$\bZ^{\alpha; W}_{T, \beta, h}$ would be defined in analogy to \eqref{eq:Zpin}.
The problem is that a.e. realization of the $\alpha$-stable regenerative set $\btau$ has
zero Lebesgue measure, hence the integral in \eqref{eq:nosense} vanishes,
yielding the ``trivial'' definition $\bP^{\alpha;W}_{T, \beta, h} = \bP^\alpha$.

These difficulties turn out to be substantial and not just technical:
as we shall see,
a non-trivial scaling limit of $\P^{\omega}_{N,\beta,h}$ does exist,
but, for $\alpha \in (\frac{1}{2},1)$, it is \emph{not} absolutely
continuous with respect to the law $\bP^\alpha$
(hence no formula like \eqref{eq:nosense} can hold).
Note that an analogous phenomenon happens for the
directed polymer in random environment, cf.~\cite{AKQ14b}.

\subsection{Main results}
\label{S:main1}

We need to formulate an additional assumption
on the renewal processes that we consider. Introducing the renewal function
\begin{equation*}
	u(n) := \P(n\in\tau) = \sum_{k=0}^\infty \P(\tau_k = n) ,
\end{equation*}
assumption \eqref{eq:ass} yields $u(n+\ell) / u(n) \to 1$ as $n\to\infty$, provided $\ell = o(n)$
(see \eqref{un} below).
We ask that the rate of this convergence is at least a power-law of $\frac{\ell}{n}$:
\begin{equation}\label{eq:renass0+}
	\exists C, n_0 \in (0,\infty), \ \epsilon,\delta \in (0,1]: \qquad
	\bigg| \frac{u(n+\ell)}{u(n)} - 1 \bigg| \le C \bigg(\frac{\ell}{n}\bigg)^\delta
	\quad \ \forall n \ge n_0, \ 0 \le \ell \le \epsilon n \,.
\end{equation}

\begin{remark}\label{R:Ken}\rm
As we discuss in Appendix~\ref{sec:dimo}, condition \eqref{eq:renass0+}
is a \emph{very mild} smoothness requirement,
that can be verified in most situations. E.g.,
it was shown by Alexander~\cite{A11} that for any $\alpha>0$ and
slowly varying $L(\cdot)$, there exists a Markov chain $X$ on $\N_0$ with $\pm1$ steps,
called Bessel-like random walk,
whose return time to $0$, denoted by $T$, is such that
\begin{equation}\label{eq:assK}
	K(n) := \P(T=2n) = \frac{\tilde L(n)}{n^{1+\alpha}} \qquad \mbox{as } n\to\infty,
	\qquad \text{with $\widetilde L(n) \sim L(n)$}.
\end{equation}
We prove in Appendix~\ref{sec:dimo} that
any such walk always satisfies \eqref{eq:renass0+}.
\end{remark}

Recall that $\cC$ denotes the compact Polish space of closed subsets of $\R$.
We denote by $\cM_1(\cC)$ the space of Borel probability
measures on $\cC$, which is itself a compact Polish space,
equipped with the topology of weak convergence.
We will work with a \emph{conditioned} version
of the disordered pinning model \eqref{eq:Gweight}, defined by
\begin{equation} \label{eq:condpm}
	\p^{\omega, \rm c}_{N,\beta,h}(\,\cdot\,) :=
	\p^{\omega}_{N,\beta,h}(\,\cdot\,| N \in \tau) \,.
\end{equation}
(In order to lighten notation,
when $N \not\in \N$ we agree that $\p^{\omega, \rm c}_{N,\beta,h} :=
\p^{\omega, \rm c}_{\lfloor N \rfloor,\beta,h}$.)

Recalling \eqref{eq:rescN}, let us introduce the notation
\begin{equation}\label{eq:dtau/N}
	\p^{\omega,\rm c}_{NT , \beta_N ,h_N}(\dd (\tau/N)) :=
	\text{ law of the rescaled set }
	\frac{\tau }{N} \cap [0,T]
	\text{ under }
	\p^{\omega,\rm c}_{NT , \beta_N ,h_N} \,.
\end{equation}
For a fixed realization of the disorder $\omega$,
$\p^{\omega,\rm c}_{NT , \beta_N ,h_N}(\dd (\tau/N))$ is
a probability law on $\cC$, i.e.\
an element of $\cM_1(\cC)$.
Since $\omega$ is chosen randomly,
the law $\p^{\omega,\rm c}_{NT , \beta_N ,h_N}(\dd (\tau/N))$
is a \emph{random} element of $\cM_1(\cC)$, i.e.\ a
$\cM_1(\cC)$-valued random variable.

Our first main result
is the convergence in distribution of this random variable,
provided $\alpha \in (\frac{1}{2},1)$ and the coupling constants $\beta = \beta_N$
and $h = h_N$ are rescaled appropriately:
\begin{equation}\label{eq:scalingbetah}
    \beta_N :=
	\hbeta\frac{L(N)}{ N^{\alpha - \frac{1}{2}}} \,,
	\qquad
	h_N := \hh\frac{L(N)}{ N^{\alpha} } \,,
    \qquad \text{for} \ N\in\N ,
    \ \hbeta > 0, \ \hh \in \R \,.
\end{equation}

\begin{theorem}[Existence and universality of the CDPM]\label{th:main1}
Fix $\alpha \in (\frac{1}{2},1)$, $T>0$, $\hbeta > 0$, $\hh\in\R$.
There exists a $\cM_1(\cC)$-valued random variable $\bP^{\alpha;W,\rm c}_{T, \hbeta, \hh}$,
called the (conditioned) \emph{continuum disordered pinning model (CDPM)},
which is a function of the parameters $(\alpha, T, \hbeta, \hh)$ and of
a standard Brownian motion $W = (W_t)_{t\ge 0}$,
with the following property:
\begin{itemize}
\item
for any renewal process $\tau$ satisfying \eqref{eq:ass} and \eqref{eq:renass0+},
and $\beta_N$, $h_N$ defined as in \eqref{eq:scalingbetah};
\item
for any i.i.d.\ sequence $\omega$ satisfying \eqref{eq:omegacond};
\end{itemize}
the law
$\p^{\omega,\rm c}_{NT , \beta_N ,h_N}(\dd (\tau/N))$ of the rescaled pinning model,
cf. \eqref{eq:dtau/N},
viewed as a $\cM_1(\cC)$-valued random variable,
converges in distribution to $\bP^{\alpha;W,\rm c}_{T, \hbeta, \hh}$ as $N\to\infty$.
\end{theorem}

We refer to Subsection~\ref{S:disrel} for a discussion on the universality
of the CDPM. We stress that the restriction $\alpha \in (\frac{1}{2},1)$
is substantial and not technical, being linked with the issue of \emph{disorder relevance},
as we explain in Subsection~\ref{S:disrel}
(see also \cite{CSZ13}).
\vskip 2mm
Let us give a quick explanation of the choice of scalings \eqref{eq:scalingbetah}. This is the canonical scaling under which the partition function $Z^\omega_{N,\gb_N,h_N}$
in \eqref{eq:Zpin} has a nontrivial continuum limit. To see this, write
\begin{align*}
Z^\go_{N,\gb,h}=\E\Big[ \prod_{n=1}^N(1+\epsilon^{\gb,h}_n\,\ind_{n\in\tau} )\Big]
=1+\sum_{k=1}^N\sum_{1\leq n_1<\cdots<n_k\leq N} \epsilon^{\gb,h}_{n_1}\cdots \epsilon^{\gb,h}_{n_k} \,\P(n_1\in\tau,...,n_k\in \tau),
\end{align*}
where $\epsilon^{\gb,h}_{n}:=e^{\gb\go_n-\Lambda(\gb)+h}-1$. By
Taylor expansion, as $\gb,h$ tend to zero, one has the asymptotic
behavior $\bbE[\epsilon^{\gb,h}_{n}] \approx h$ and $\bbvar(\epsilon^{\gb,h}_{n})
\approx \gb^2$.
Using this fact,
we see that the asymptotic mean and variance behavior of the first term
($k=1$) in the above series is
\begin{align*}
\bbE\Big[\sum_{n=1}^N \epsilon^{\gb,h}_{n} \P(n\in \tau) \,\Big]
& \approx h \sum_{n=1}^N  \P(n\in \tau) \approx h\frac{N^\ga}{L(N)} \,, \\
\bbvar\Big[\sum_{n=1}^N \epsilon^{\gb,h}_{n} \P(n\in \tau) \,\Big]
& \approx \gb^2 \sum_{n=1}^N  \P(n\in \tau)^2 \approx \gb^2\frac{N^{2\ga-1}}{L(N)^2} \,,
\end{align*}
because $\P(n\in \tau) \approx n^{\alpha-1} / L(n)$, by \eqref{eq:ass}
(see \eqref{un} below).
Therefore, for these quantities to have a non-trivial limit as $N$ tends to infinity, we are forced to scale $\beta_N$ and $h_N$ as in \eqref{eq:scalingbetah}. Remarkably, this is also the correct scaling for higher order terms in the expansion for $Z^\go_{N,\gb,h}$,
as well as for the measure $\P^\go_{N,\gb_N,h_N}$ to converge to a non-trivial limit.

\smallskip

We now describe the continuum measure.
For a fixed realization of the Brownian motion $W = (W_t)_{t\in [0,\infty)}$,
which represents the ``continuum disorder'',
we call $\bP^{\alpha; W,\rm c}_{T, \hbeta, \hh}$
the \emph{quenched law} of the CDPM, while
\begin{equation}\label{eq:averaged}
	\bbE \Big[ \bP^{\alpha; W,\rm c}_{T, \hbeta, \hh} \Big] (\,\cdot\,)
	:= \int \bP^{\alpha; W,\rm c}_{T, \hbeta, \hh} (\,\cdot\,) \, \bbP(\dd W)
\end{equation}
will be called the \emph{averaged law} of the CDPM.
We also introduce, for $T > 0$,
the law $\bP_{T}^{\alpha; \rm c}$ of the $\alpha$-stable regenerative set $\btau$ restricted on $[0,T]$
and conditioned to visit $T$:
\begin{equation}\label{eq:bPnT}
	\bP_{T}^{\alpha; \rm c}(\,\cdot\,) := \bP^\alpha( \btau \cap [0,T] \in \,\cdot\,|\, T \in \btau) \,,
\end{equation}
which will be called the
\emph{reference law}.
(Relation \eqref{eq:bPnT} is defined through regular conditional distributions.)
Note that both $\bbE\big[\bP^{\alpha; W,\rm c}_{T, \hbeta, \hh}\big]$ and $\bP_{T}^{\alpha; \rm c}$
are probability laws on $\cC$, while
$\bP^{\alpha; W,\rm c}_{T, \hbeta, \hh}$ is a \emph{random}
probability law on $\cC$.

\smallskip

Intuitively, the quenched law $\bP^{\alpha; W,\rm c}_{T, \hbeta, \hh}$
could be conceived as a ``Gibbs transformation'' of the reference law $\bP_{T}^{\alpha; \rm c}$,
where each visited site $t \in \btau \cap [0,T]$
of the $\alpha$-stable regenerative set is given
a reward/penalty $\hbeta \frac{\dd W_t}{\dd t} + \hh$, like in the discrete case.
This heuristic interpretation should be taken with care, however, as the following results show.

\begin{theorem}[Absolute Continuity of the Averaged CDPM]\label{T:averageabs}
For all $\alpha \in (\frac{1}{2},1)$, $T>0$, $\hbeta > 0$, $\hh\in\R$,
the averaged law
$\bbE\big[\bP^{\alpha; W, \rm c}_{T, \hat \beta, \hat h}\big]$ of the CDPM
is absolutely continuous with respect to the
reference law $\bP^{\alpha; \rm c}_{T}$.
It follows that any typical property of the reference law $\bP^{\alpha; \rm c}_{T}$
is also a typical property
of the quenched law $\bP^{\alpha; W, \rm c}_{T, \hat \beta, \hat h}$, for a.e.\
realization of $W$:
\begin{equation} \label{eq:abscont}
	\forall A \subseteq \cC \text{ such that } \bP^{\alpha; \rm c}_{T}(A) = 1: \qquad
	\bP^{\alpha; W, \rm c}_{T, \hat \beta, \hat h}(A) = 1 \text{ for $\bbP$-a.e. $W$} \,.
\end{equation}
In particular, for a.e. realization of $W$, the quenched law
$\bP^{\alpha; W, \rm c}_{T, \hat \beta, \hat h}$ of the CDPM is supported on closed
subsets of $[0,T]$ with Hausdorff dimension $\alpha$.
\end{theorem}

It is tempting to deduce from \eqref{eq:abscont} the absolute continuity of the
quenched law $\bP^{\alpha; W, \rm c}_{T, \hat \beta, \hat h}$ with
respect to the reference law $\bP^{\alpha; \rm c}_{T}$, for a.e.\
realization of $W$, \emph{but this is false}.

\begin{theorem}[Singularity of the Quenched CDPM]\label{T:quenchsing}
For all $\alpha \in (\frac{1}{2},1)$, $T>0$, $\hbeta > 0$, $\hh\in\R$
and for a.e.\ realization of $W$, the
quenched law $\bP^{\alpha; W, \rm c}_{T, \hat \beta, \hat h}$ of the CDPM
is singular with respect to the reference law $\bP^{\alpha; \rm c}_{T}$:
\begin{equation} \label{eq:sing}
        \text{ for $\bbP$-a.e. $W$, } \
	\exists A \subseteq \cC \ \text{ such that } \ \bP^{\alpha; \rm c}_{T}(A) = 1
 	\ \text{ and } \
	\bP^{\alpha; W, \rm c}_{T, \hat \beta, \hat h}(A) = 0 \,.
\end{equation}
\end{theorem}

The seeming contradiction between \eqref{eq:abscont} and \eqref{eq:sing}
is resolved noting that in \eqref{eq:abscont} one cannot exchange ``$\forall A \subseteq \cC$'' and
``for $\bbP$-a.e. $W$'', because there are uncountably
many $A \subseteq \cC$ (and, of course, the set $A$ appearing in \eqref{eq:sing}
depends on the realization of $W$).

\medskip

We conclude our main results with an explicit characterization of the CDPM.
As we discuss in Appendix~\ref{A:RCS},
each closed subset $C \subseteq \R$ can be identified with two
non-decreasing and right-continuous functions
$\tg_t(C)$ and $\td_t(C)$, defined for $t\in\R$ by
\begin{equation} \label{eq:gtdt0}
\tg_t(C) := \sup\{x: \
x\in C\cap [-\infty, t]\}, \qquad \td_t(C) := \inf\{x: \
x\in C \cap (t, \infty]\} \,.
\end{equation}
As a consequence, the law of a \emph{random} closed subset $X \subseteq \R$
is uniquely determined by the finite dimensional distributions of
the random functions $(\tg_t(X))_{t\in\R}$ and $(\td_t(X))_{t\in\R}$,
i.e.\ by the probability laws on $\bR^{2k}$ given,
for $k \in \N$ and $-\infty < t_1 < t_2 < \ldots < t_k < \infty$, by
\begin{equation} \label{eq:fdd0}
	\p\big( \tg_{t_1}(X) \in \dd x_1, \, \td_{t_1}(X)
	\in \dd y_1, \, \ldots, \, \tg_{t_k}(X) \in \dd x_k, \, \td_{t_k}(X) \in \dd y_k \big) \,.
\end{equation}
As a further simplification, it is enough to focus on the event that
$X \cap [t_i, t_{i+1}] \ne \emptyset$ for all $i=1,\ldots, k$, that is, one can
restrict $(x_1, y_1, \ldots, x_k, y_k)$ in \eqref{eq:fdd0} on the following set:
\begin{equation} \label{eq:restr}
	\begin{split}
	\cR_{t_0, \ldots, t_{k+1}}^{(k)} := \big\{(x_1, y_1, \ldots, x_k, y_k) : \
	& \ x_i \in [t_{i-1},t_i], \ y_i \in [t_i,t_{i+1}]
	\text{ for } i=1,\ldots, k, \\
	& \text{ such that } y_{i} \le x_{i+1} \text{ for } i=1,\ldots, k-1  \, \big\} ,
\end{split}
\end{equation}
with $t_0 = -\infty$ and $t_{k+1} := +\infty$.
The measures \eqref{eq:fdd0} restricted on the set \eqref{eq:restr}
will be called \emph{restricted finite-dimensional distributions} (f.d.d.) of the random set $X$
(see \S\ref{sec:restricted}).

\smallskip

We can characterize the CDPM by
specifying its restricted f.d.d.. We need two ingredients:
\begin{enumerate}
\item The restricted f.d.d. of the $\alpha$-stable regenerative set conditioned to visit $T$, i.e.\
of the reference law $\bP_T^{\alpha;\rm c}$ in \eqref{eq:bPnT}:
by Proposition~\ref{P:univrenew}, these are
absolutely continuous with respect to the Lebesgue measure
on $\R^{2k}$, with the following density (with $y_{0} := 0$):
\begin{align}\label{fddalpha}
\mathsf{f}_{T;t_1,\ldots,t_k}^{\alpha;\rm c}(x_1,y_1,\ldots,x_k,y_k)
= & \Bigg( \prod_{i=1}^{k}
\frac{C_\alpha}{(x_{i} - y_{i-1})^{1-\alpha}
\, (y_i-x_i)^{1+\alpha}}
\Bigg) \frac{T^{1-\alpha}}{(T-y_k)^{1-\alpha}} \,, \\
	\label{eq:Calpha}
	\text{with}
	\quad C_\alpha := & \,\frac{\alpha \sin(\pi\alpha)}{\pi} \,,
\end{align}
where we restrict $(x_1, y_1, \ldots, x_k, y_k)$ on the set \eqref{eq:restr},
with $t_0 = 0$ and $t_{k+1} := T$.
\item A family of \emph{continuum partition functions} for our model:
\begin{equation*}
	\big(\bZ^{\alpha; W, \rm c}_{\hat\beta, \hat h}(s, t)\big)_{0\leq s\leq t <\infty} \,.
\end{equation*}
These were constructed
in \cite{CSZ13} as the limit, in the sense of finite-dimensional distributions,
of the following discrete family (under an appropriate rescaling):
\begin{equation}\label{eq:Zpinc}
	Z_{\gb,h}^{\go,\rm c}(a,b):=\E\Big[ e^{\sum_{n=a+1}^{b-1}
	(\beta \omega_n	-\Lambda(\beta) + h) \ind_{\{n\in\tau\}}}\,\Big| a\in\tau, \,
	b \in \tau \Big], \qquad 0\leq a\leq b \,.
\end{equation}
In Section~\ref{S:convZ} we upgrade the f.d.d. convergence to the process level,
deducing important a.s. properties,
such as strict positivity and continuity
(cf. Theorems~\ref{T:cconv2} and~\ref{T:bZprop}).
\end{enumerate}

We can finally characterize the restricted f.d.d. of the CDPM as follows.

\begin{theorem}[F.d.d. of the CDPM]\label{T:cconv3}
Fix $\alpha \in (\frac{1}{2},1)$, $T>0$,
$\hbeta > 0$, $\hh\in\R$ and let
$\big(\bZ^{\alpha; W, \rm c}_{\hat\beta, \hat h}(s, t)\big)_{0\leq s\leq t < \infty}$
be an a.s. continuous version of the continuum partition functions.
For a.e.\ realization of $W$, the quenched law $\bP^{\alpha; W, \rm c}_{T, \hat \beta, \hat h}$
of the CDPM (cf. Theorem~\ref{th:main1})
can be defined as the unique probability law
on $\cC$ which satisfies the following properties:
\begin{ienumerate}
\item\label{it:1b} $\bP^{\alpha; W, \rm c}_{T, \hat \beta, \hat h}$
is supported on closed subsets $\btau \subseteq [0,T]$ with $\{0, T\} \subseteq \btau$.

\item\label{it:2b} For all $k\in\N$ and $0 =: t_0 <t_1<\cdots <t_k < t_{k+1} := T$,
and for $(x_1, y_1, \ldots, x_k, y_k)$ restricted on the set
$\cR^{(k)}_{t_0,\ldots, t_{k+1}}$ in \eqref{eq:restr},
the f.d.d. of $\bP^{\alpha; W, \rm c}_{T, \hat \beta, \hat h}$ have densities given by
\begin{equation}\label{CDPMfdd}
\begin{aligned}
& \frac{\bP^{\alpha; W, \rm c}_{T, \hat \beta, \hat h}\big(\tg_{t_1}(\btau) \in \dd x_1,\,
\td_{t_1}(\btau) \in \dd y_1, \ldots, \tg_{t_k}(\btau)\in \dd x_k,\, \td_{t_k}(\btau) \in \dd y_k\big)}
{\dd x_1\, \dd y_1 \, \cdots \dd x_k\, \dd y_k} \\
& \ \ =\,
\Bigg(\, \frac{\prod_{i=0}^k \bZ^{\alpha; W, \rm c}_{\hat \beta, \hat h}(y_{i},x_{i+1})}
{\bZ^{\alpha; W, \rm c}_{\hat \beta, \hat h}(0,T)} \Bigg)\,
\mathsf{f}_{T;t_1,\ldots,t_k}^{\alpha;\rm c}(x_1,y_1,\ldots,x_k,y_k) \,,
\end{aligned}
\end{equation}
where we set $y_0 := 0$ and $x_{k+1} := T$, and where
$\mathsf{f}_{T;t_1,\ldots,t_k}^{\alpha;\rm c}(\cdot)$
is defined in \eqref{fddalpha}.
\end{ienumerate}
\end{theorem}

\subsection{Discussion and perspectives}

\label{S:disrel}

We conclude the introduction with some observations on the results
stated so far, putting them in the context of the existing literature,
stating some conjectures and outlining further directions of research.

\medskip
\smallskip
\noindent\textbf{1. (Disorder relevance).}
The parameter $\beta$ tunes the strength of the disorder
in the model $\p^{\omega, \rm c}_{N,\beta,h}$,
cf.\ \eqref{eq:condpm}, \eqref{eq:Gweight}.
When $\beta=0$, the sequence $\omega$ disappears and we obtain the
so-called \emph{homogeneous pinning model}.
Roughly speaking, the effect of disorder is said to be:
\begin{itemize}
\item  {\em irrelevant} if the disordered model
($\beta > 0$) has the same qualitative behavior
as the homogeneous model ($\beta = 0$),
provided the disorder is sufficiently weak ($\beta \ll 1$);
\item {\em relevant} if, on the other hand, an arbitrarily small amount of disorder (any $\beta > 0$)
alters the qualitative behavior of the homogeneous model ($\beta = 0$).
\end{itemize}
Recalling that $\alpha$
is the exponent appearing in \eqref{eq:ass}, it is known that
disorder is irrelevant for pinning models when $\alpha < \frac{1}{2}$ and relevant when
$\alpha > \frac{1}{2}$, while the case $\alpha = \frac{1}{2}$ is called marginal and is more delicate
(see~\cite{G10} and the  references therein for an overview).

\smallskip

It is natural to interpret our results from this perspective.
For simplicity, in the sequel we set $h_N := \hh \, L(N)/N^\alpha$,
as in \eqref{eq:scalingbetah},
and we use the notation $\p^{\omega,\rm c}_{NT , \beta_N ,h_N}(\dd (\tau/N))$,
cf.\ \eqref{eq:dtau/N}, for
the law of the rescaled set $\tau/N$ under the pinning model.

\smallskip

In the homogeneous case ($\beta = 0$),
it was shown in
\cite[Theorem~3.1]{Soh09}\footnote{Actually \cite{Soh09}
considers the non-conditioned case \eqref{eq:Gweight},
but it can be adapted to the conditioned case.} that the weak limit of
$\p^{\alpha; \rm c}_{NT ,0,h_N}(\dd (\tau/N))$ as $N\to\infty$
is a probability law $\bP_{T, 0, \hh}^{\alpha; \rm c}$
on $\cC$ which is \emph{absolutely continuous} with respect to the reference law
$\bP_{T}^{\alpha; \rm c}$ (recall \eqref{eq:bPnT}):
\begin{equation}\label{eq:Sohier}
	\frac{\dd \bP_{T, 0, \hh}^{\alpha; \rm c}}{\dd \bP_{T}^{\alpha; \rm c}}(\btau) =
	\frac{e^{\hh \cL_T(\btau)}}{\bE[e^{\hh \cL_T(\btau)}]} \,,
\end{equation}
where $\cL_T(\btau)$ denotes the so-called local time associated
to the regenerative set $\btau$.
We stress that this result holds with \emph{no restriction on $\alpha \in (0,1)$}.

Turning to the disordered model $\beta > 0$,
what happens for $\alpha \in (0,\frac{1}{2})$?
In analogy with \cite{B89,CY06}, we conjecture
that \emph{for fixed $\beta > 0$ small enough},
the limit in distribution of
$\P_{NT, \beta, h_N}^{\omega, \rm c}(\dd (\tau/N))$
as $N\to\infty$
\emph{is the same as for the homogeneous model} ($\beta = 0$), i.e.\ the
law $\bP_{T, 0, \hh}^{\alpha; \rm c}$ defined in \eqref{eq:Sohier}.
Thus, for $\alpha \in (0,\frac{1}{2})$, the continuum model
is \emph{non-disordered} (deterministic)
and absolutely continuous with respect to the reference law.

This is in striking contrast with the case $\alpha \in (\frac{1}{2},1)$, where our results show that
the continuum model $\bP_{T, \hbeta, \hh}^{\alpha; W, \rm c}$ is truly \emph{disordered} and
singular with respect to the reference law (cf. Theorems~\ref{th:main1}, \ref{T:averageabs},
\ref{T:quenchsing}).
In other terms, for $\alpha \in (\frac{1}{2},1)$,
\emph{disorder survives in the scaling limit}
(even though $\beta_N, h_N \to 0$) and breaks down the absolute continuity with
respect to the reference law,
providing a clear manifestation of disorder relevance.

We refer to \cite{CSZ13} for a general discussion on disorder relevance in our framework.

\medskip
\smallskip
\noindent\textbf{2. (Universality).}
The quenched law $\bP_{T, \hbeta, \hh}^{\alpha; W, \rm c}$ of the CDPM
is a \emph{random probability law on $\cC$}, i.e. a random variable
taking values in $\cM_1(\cC)$.
Its distribution is a probability law on the space $\cM_1(\cC)$
---i.e. an element of $\cM_1(\cM_1(\cC))$--- which
is \emph{universal}: it depends on
few macroscopic parameters (the time horizon $T$, the disorder strength and bias $\hbeta, \hh$ and
the exponent $\alpha$) but not on finer details of the discrete model
from which it arises, such as the distributions of $\omega_1$ and of $\tau_1$:
all these details disappear in the scaling limit.

Another important universal aspect of the CDPM is linked to
\emph{phase transitions}.
We do not explore this issue here,
referring to \cite[\S1.3]{CSZ13} for a detailed discussion,
but we mention that the CDPM leads to \emph{sharp predictions}
about the asymptotic behavior of the \emph{free energy} and \emph{critical curve}
of discrete pinning models, in the weak disorder regime $\lambda, h \to 0$.

\medskip
\noindent\textbf{3. (Bessel processes).}
In this paper we consider pinning models built on top of general renewal
processeses $\tau = (\tau_k)_{k\in\N_0}$ satisfying \eqref{eq:ass} and \eqref{eq:renass0+}.
In the special case when the renewal
process is the zero level set of a Bessel-like random walk \cite{A08}
(recall Remark~\ref{R:Ken}), one can define the pinning model
\eqref{eq:Gweight}, \eqref{eq:condpm} as a probability law on random walk paths
(and not only on their zero level set).

Rescaling the paths diffusively,
one has an analogue of Theorem~\ref{th:main1},
in which the CDPM is built as a random probability
law on the space $C([0,T], \R)$ of continuous functions from $[0,T]$ to $\R$.
Such an extended CDPM is a continuous process $(\bX_t)_{t\in [0,T]}$, that
can be heuristically described as
a \emph{Bessel process of dimension $\delta =2(1-\alpha)$ interacting
with an independent Brownian environment $W$ each time $X_t = 0$}.
The ``original'' CDPM of our Theorem~\ref{th:main1}
corresponds to the zero level set $\btau := \{t\in[0,T]: \bX_t = 0\}$.

We stress that, starting from the zero level set $\btau$,
one can reconstruct the whole process $(\bX_t)_{t\in [0,T]}$
by pasting independent Bessel excusions on top of $\btau$
(more precisely, since the open set $[0,T] \setminus \btau$ is a countable
union of disjoint open intervals, one attaches a Bessel excursion to each of these
intervals).\footnote{Alternatively,
one can write down explicitly the f.d.d. of
$(\bX_t)_{t\in [0,T]}$ in terms of the continuum partition functions
$\bZ^{\alpha; W, \rm c}_{\hat \beta, \hat h}(s,t)$ (see Section~\ref{S:convZ}).
We skip the details for the sake of brevity.}
This provides a rigorous definition of $(\bX_t)_{t\in [0,T]}$
in terms of $\btau$ and shows that the zero level set is
indeed the fundamental object.

\medskip
\noindent\textbf{4. (Infinite-volume limit).}
Our continuum model
$\bP_{T, \hbeta, \hh}^{\alpha; W, \rm c}$ is built on a finite interval $[0,T]$.
An interesting open problem is to let $T \to \infty$, proving that
$\bP_{T, \hbeta, \hh}^{\alpha; W, \rm c}$ converges
in distribution to an \emph{infinite-volume CDPM} $\bP_{\infty, \hbeta, \hh}^{\alpha; W, \rm c}$.
Such a limit law would inherit scaling properties from the continuum partition functions,
cf.\ Theorem~\ref{T:bZprop}~\eqref{it:3}.
(See also \cite{RVY08} for related work in the non-disordered case $\hbeta = 0$.)

\subsection{Organization of the paper}

The rest of the paper is organized as follows.
\begin{itemize}
\item In Section~\ref{S:convZ}, we study the properties of continuum partition functions.

\item In Section~\ref{S:CDPMchar}, we prove Theorem~\ref{T:cconv3} on the characterization of the
CDPM, which also yields Theorem~\ref{th:main1}.

\item In Section~\ref{S:CDPMprop}, we prove Theorems~\ref{T:averageabs} and~\ref{T:quenchsing}
on the relations between the CDPM and the $\alpha$-stable regenerative set.

\item In Appendix~\ref{A:RCS}, we
describe the measure-theoretic background needed to study
random closed subsets of $\R$, which is of independent interest.

\item Lastly,
in Appendices \ref{app:renew_est}
and \ref{app:int_est} we prove some auxiliary estimates.
\end{itemize}

\section{Continuum partition functions as a process}
\label{S:convZ}

In this section we focus on a family
$\big(\bZ^{\alpha; W, \rm c}_{\hat\beta, \hat h}(s, t)\big)_{0\leq s\leq t <\infty}$
of \emph{continuum partition functions} for our model,
which was recently introduced
in \cite{CSZ13} as the limit of the discrete family \eqref{eq:Zpinc}
in the sense of finite-dimensional distributions.
We upgrade this convergence to the process level (Theorem~\ref{T:cconv2}),
which allows us to deduce important properties (Theorem~\ref{T:bZprop}).
Besides their own interest, these results are the key to the
construction of the CDPM.

\subsection{Fine properties of continuum partition functions}

Recalling \eqref{eq:Zpinc}, where $Z_{\gb,h}^{\go, \rm c}(a, b)$ is defined
for $a,b\in\N_0$, we extend
$Z_{\gb,h}^{\go, \rm c}(\cdot, \cdot)$
to a continuous function on
$$
[0,\infty)^2_\le :=\{(s,t) \in [0,\infty)^2: \ 0\leq s\leq t <\infty\} \,,
$$
bisecting each unit square
$[m-1,m]\times [n-1,n]$, with $m\leq n\in \N$, along the
main diagonal and
linearly interpolating $Z_{\gb,h}^{\go, \rm c}(\cdot, \cdot)$ on each triangle.
In this way, we can regard
\begin{equation} \label{eq:contfam}
	\big(Z^{\omega,\rm c}_{\beta_N,h_N}(sN, tN)\big)_{0\leq s\leq t <\infty}
\end{equation}
as random variables taking values in the space $C([0,\infty)^2_\le, \R)$,
equipped with the topology of uniform convergence on compact sets and with the corresponding
Borel $\sigma$-algebra.
The randomness comes from the disorder sequence $\omega = (\omega_n)_{n\in\N}$.

Even though our main interest in this paper
is for $\alpha \in (\frac{1}{2},1)$, we also include the case
$\alpha > 1$ in the following key result, which is proved in
Subsection~\ref{S:proofcconv2} below.

\begin{theorem}[Process Level Convergence of Partition Functions]\label{T:cconv2}
Let $\alpha\in (\frac{1}{2}, 1)\cup (1,\infty)$,
$\hat\beta > 0$, $\hat h\in\R$.
Let $\tau$ be a renewal process satisfying \eqref{eq:ass}
and \eqref{eq:renass0+}, and $\omega$ be an
i.i.d.\ sequence satisfying \eqref{eq:omegacond}.
For every $N\in\N$, define $\beta_N, h_N$ by (recall \eqref{eq:scalingbetah})
\begin{equation}\label{eq:scalingbetah>1}
	\begin{cases}
	\displaystyle\beta_N :=
	\hbeta\frac{L(N)}{ N^{\alpha - \frac{1}{2}}} \\
	\rule{0pt}{2em}\displaystyle h_N := \hh\frac{L(N)}{ N^{\alpha} }
    \end{cases}
    \ \text{for } \alpha \in (\tfrac{1}{2},1) \,,
    \qquad
	\begin{cases}
	\displaystyle\beta_N := \frac{\hbeta}{\sqrt{N}} \\
	\rule{0pt}{2em}\displaystyle h_N := \frac{\hh}{N}
	\end{cases}
	\ \text{for } \alpha > 1 \,.
\end{equation}
As $N\to\infty$ the two-parameter family
$\big(Z^{\omega,\rm c}_{\beta_N,h_N}(sN, tN) \big)_{0\leq s\leq t <\infty}$
converges in distribution on $C([0,\infty)^2_\le, \R)$ to a family
$\big(\bZ^{\alpha; W, \rm c}_{\hat\beta, \hat h}(s, t)\big)_{0\leq s\leq t <\infty}$,
called \emph{continuum partition functions}.
For all $0\leq s\leq t<\infty$, one has the Wiener chaos representation
\begin{equation}\label{cpinWC}
\bZ^{\alpha; W, \rm c}_{\hat\beta, \hat h}(s, t)
= 1 + \sum_{k=1}^\infty \ \ \idotsint\limits_{s<t_1<\cdots <t_k<t}
\bpsi^{\alpha; \rm c}_{s,t}(t_1,\ldots, t_k)
\prod_{i=1}^k (\hat \beta \, \dd W_{t_i} + \hat h \, \dd t_i) \,,
\end{equation}
where $W = (W_t)_{t\ge 0}$ is a standard Brownian motion,
the series in \eqref{cpinWC} converges in $L^2$, and
the kernel $\bpsi^{\alpha; \rm c}_{s,t}(t_1,\ldots, t_k)$ is defined as follows,
with $C_\alpha$ as in \eqref{eq:Calpha} and $t_0 := s$:
\begin{equation}\label{cpinWCphi}
\bpsi^{\alpha; \rm c}_{s,t}(t_1,\ldots, t_k) =
\begin{cases}
\displaystyle
\Bigg(\prod_{i=1}^k \frac{C_\alpha}{(t_i - t_{i-1})^{1-\alpha}} \Bigg)
\frac{(t-s)^{1-\alpha}}{(t-t_k)^{1-\alpha}}
& \text{if  } \alpha \in ( \tfrac{1}{2}, 1), \\
\rule{0pt}{1.8em}\displaystyle
\hfill\frac{1}{\e[\tau_1]^k}\hfill & \text{if  } \alpha >1.
\end{cases}
\end{equation}
\end{theorem}

\begin{remark}\rm\label{R:CM}
The integral in \eqref{cpinWC}
is defined by expanding formally the product of differentials
and reducing to standard multiple Wiener and Lebesgue integrals.
An alternative equivalent definition is to note that, by Girsanov's theorem,
the law of $(\hbeta W_t +\hh  t)_{t\in [0,T]}$ is
absolutely continuous w.r.t.\ that of $(\hbeta W_t)_{t\in [0,T]}$, with Radon-Nikodym density
\begin{equation} \label{eq:RN}
	\mathfrak{f}_{T, \hbeta, \hh}(W) :=
	e^{(\frac{\hat h}{\hat \beta}) \, W_T -
	\frac{1}{2} (\frac{\hat h}{\hat \beta})^2 T} \,.
\end{equation}
It follows that $\big(\bZ^{\alpha; W, \rm c}_{\hat\beta, \hat h}(s, t)\big)_{0\leq s\leq t\leq T}$
has the same law as $\big(\bZ^{\alpha; W, \rm c}_{\hat\beta, 0}(s, t)\big)_{0\leq s\leq t\leq T}$
(for $\hh = 0$) under
a change of measure with density \eqref{eq:RN}. For further details, see~\cite{CSZ13}.
\end{remark}

\begin{remark}\rm
Theorem~\ref{T:cconv2}
still holds if we also include the two-parameter family
of \emph{unconditioned} partition functions
$\big(Z^{\omega}_{\beta_N,h_N}(sN, tN)\big)_{0\leq s\leq t<\infty}$,
defined the same way as $Z^{\omega, \rm c}_{\beta, h}(a,b)$ in \eqref{eq:Zpinc}, except
for removing the conditioning on $b\in \tau$. The limiting process
$\bZ^{\alpha; W}_{\hat\beta, \hat h}(s, t)$ will then have a kernel $\bpsi_{s,t}^{\alpha}$, which
modifies $\bpsi^{\alpha; \rm c}_{s,t}$ in \eqref{cpinWCphi}, by setting
\begin{equation}\label{pinWCphi}
\bpsi^{\alpha}_{s,t}(t_1,\ldots, t_k) =
\prod_{i=1}^k \frac{C_\alpha}{(t_i - t_{i-1})^{1-\alpha}}, \qquad \text{if  }
\alpha \in (\tfrac{1}{2}, 1).
\end{equation}
\end{remark}

\smallskip

By Theorem~\ref{T:cconv2}, we can fix a version of the continuum partition functions
$\bZ^{\alpha; W, \rm c}_{\hat\beta, \hat h}(s, t)$
which is continuous in $(s,t)$.
This will be implicitly done henceforth.
We can then state some fundamental properties,
proved in Subsection~\ref{S:proofbZprop}.

\begin{theorem}[Properties of Continuum Partition Functions]\label{T:bZprop}
For all $\alpha \in (\tfrac{1}{2}, 1)$, $\hat\beta>0$, $\hat h\in \R$ the following properties hold:
\begin{ienumerate}
\item\label{it:1} \emph{(Positivity)}
For a.e.\ realization of $W$, the function
$(s,t) \mapsto \bZ^{\alpha; W, \rm c}_{\hat\beta, \hat h}(s,t)$ is continuous and strictly
positive at all $0\leq s\leq t < \infty$.

\item\label{it:2} \emph{(Translation Invariance)} For any
fixed $t>0$,
the process $(\bZ^{\alpha; W, \rm c}_{\hat\beta, \hat h}(t,t+u))_{u\geq 0}$ has the same distribution as
$(\bZ^{\alpha; W, \rm c}_{\hat\beta, \hat h}(0,u))_{u\geq 0}$, and is independent of
$(\bZ^{\alpha; W, \rm c}_{\hat\beta, \hat h}(s,u))_{0\leq s\leq u\leq t}$.

\item\label{it:3} \emph{(Scaling Property)} For any constant $A>0$,
one has the equality in distribution
\begin{equation}\label{bZscaling}
\big(\bZ^{\alpha; W, \rm c}_{\hat\beta, \hat h}(As, At)\big)_{0\leq s\leq t <\infty}
\stackrel{\rm dist}{=} \Big(\bZ^{\alpha; W, \rm c}_{A^{\alpha-1/2}\hat\beta, A^\alpha\hat h}(s, t)
\Big)_{0\leq s\leq t < \infty} \,.
\end{equation}

\item\label{it:4} \emph{(Renewal Property)}
Setting $\bZ(s,t) := \bZ^{\alpha; W, c}_{\hat\beta, \hat h}(s,t)$
for simplicity,
for a.e. realization of $W$ one has,
for all $0\leq s<u<t < \infty$,
\begin{equation}\label{bUdecomp}
\frac{C_\alpha\,\bZ(s,t)}{(t-s)^{1-\alpha}}
= \int_{x\in (s,u)} \int_{y\in (u,t)}
\frac{C_\alpha\,\bZ(s,x)}{(x-s)^{1-\alpha}} \,
\frac{1}{(y-x)^{1+\alpha}} \, \frac{C_\alpha\,
\bZ(y,t)}{(t-y)^{1-\alpha}}\,
\dd x \, \dd y \,,
\end{equation}
which can be rewritten, recalling \eqref{fddalpha}, as follows:
\begin{equation} \label{eq:bUdecomp+}
	\bZ(s,t) =
	\bE^{\alpha; c}_{t-s} \Big[ \bZ \big( s,\tg_u(\btau) \big) \,
	\bZ \big( \td_u(\btau),t \big) \Big] \,.
\end{equation}
\end{ienumerate}
\end{theorem}

\smallskip
The rest of this section is devoted to the proof of Theorems~\ref{T:cconv2}
and~\ref{T:bZprop}.
We recall that assumption \eqref{eq:ass} entails the following key renewal estimates,
with $C_\alpha$ as in \eqref{eq:Calpha}:
\begin{equation}\label{un}
	u(n) \,:=\, \p(n \in \tau) \,\sim\,
	\begin{cases}
	\displaystyle
	 \frac{C_\alpha}{L(n) n^{1-\alpha}} & \text{ if } 0 < \alpha < 1, \\
	\rule{0pt}{1.9em}\displaystyle
	 \frac{1}{\e[\tau_1]} = (const.) \in (0,\infty) & \text{ if } \alpha > 1 ,
	\end{cases}
\end{equation}
by the classical renewal theorem for $\alpha > 1$ and by~\cite{GL63,D97}
for $\alpha \in (0,1)$.
Let us also note that the additional assumption \eqref{eq:renass0+} for $\alpha \in (0,1)$
can be rephrased as follows:
\begin{equation}\label{renewalbd1}
	|u(q) - u(r)| \leq C \Big(\frac{r-q}{r}\Big)^\delta u(q) \,,
	\qquad \forall r \ge q \ge n_0 \ \text{ with } \ r-q \le \epsilon r \,,
\end{equation}
up to a possible change of the constants $C, n_0, \epsilon$.

\subsection{Proof of Theorem~\ref{T:cconv2}}
\label{S:proofcconv2}

We may assume $T=1$. For convergence in distribution on $C([0,1]^2_\le,\R)$
it suffices to show that
$\{(Z^{\omega,\rm c}_{\beta_N,h_N}(sN, tN))_{0\leq s\leq t\leq 1}\}_{N\in\N}$ is a tight family,
because the finite-dimensional
distribution convergence was already obtained in \cite{CSZ13}
(see Theorem~3.1 and Remark~3.3 therein). We break down the proof into five steps.
\medskip

\noindent
{\bf Step 1. Moment criterion.} We recall a moment criterion for the H\"older continuity
of a family of multi-dimensional stochastic processes, which was also used in~\cite{AKQ14a} to
prove similar tightness results for the directed polymer model. Using Garsia's
inequality~\cite[Lemma~2]{G72} with $\Psi(x)=|x|^p$ and $\varphi(u)= u^q$ for $p\geq 1$
and $pq>2d$, the modulus of continuity of a continuous function $f:[0,1]^d \to \R$
can be controlled by
$$
|f(x)-f(y)| \leq 8 \int_0^{|x-y|} \Psi^{-1}\Big(\frac{B}{u^{2d}}\Big) \dd \varphi(u)
= 8 \int_0^{|x-y|} \frac{B^{1/p}}{u^{2d/p}} \dd (u^q)
= \frac{8 B^{1/p}q}{q-2d/p} |x-y|^{q -2d/p},
$$
where
$$
B=B(f) =\iint_{[0,1]^d\times [0,1]^d}
\Psi\Big(\frac{f(x)-f(y)}{\varphi\big(\frac{x-y}{\sqrt d}\big)} \Big) \dd x\, \dd y
= d^{q/2} \iint_{[0,1]^d\times [0,1]^d} \frac{|f(x)-f(y)|^p}{|x-y|^{pq}}  \dd x\, \dd y .
$$
Suppose now that $(f_N)_{N\in\N}$ are \emph{random} continuous function
on $[0,1]^d$ such that
\begin{gather*}
\E[|f_N(x)-f_N(y)|^p] \leq C|x-y|^\eta,
\end{gather*}
for some $C, p, q, \eta \in(0,\infty)$ with $pq > 2d$ and $\eta > pq-d$,
uniformly in $N\in\N$, $x,y \in [0,1]^d$.
Then $\E[B(f_N)]$ is bounded uniformly in $N$, hence $\{B(f_N)\}_{N\in\N}$ is tight.
If the functions $f_N$ are equibounded at some point
(e.g.\ $f_N(0) = 1$ for every $N\in\N$),
the tightness of $B(f_N)$ entails the tightness of $\{f_N\}_{N\in\N}$,
by the Arzel\`a-Ascoli theorem
\cite[Theorem 7.3]{B99}.

To prove the tightness of
$\{(Z^{\omega,\rm c}_{\beta_N,h_N}(sN, tN))_{0\leq s\leq t\leq 1}\}_{N\in\N}$,
it then suffices to show that
\begin{equation}\label{gamom1}
\bbE\Big[\big|Z^{\omega, \rm c}_{\beta_N, h_N}(s_1N, t_1N) - Z^{\omega, \rm c}_{\beta_N, h_N}
(s_2N, t_2N)\big|^p \Big] \leq C \big(\sqrt{(s_1-s_2)^2+(t_1-t_2)^2}\,\big)^\eta,
\end{equation}
which by triangle inequality, translation invariance and symmetry can be reduced to
\begin{equation}\label{gamom2}
\exists C > 0, \ p \ge 1, \ \eta > 2: \qquad
\bbE\Big[\big|Z^{\omega, \rm c}_{\beta_N, h_N}(0, tN) - Z^{\omega, \rm c}_{\beta_N, h_N}(0, sN)
\big|^p \Big] \leq C |t-s|^\eta ,
\end{equation}
uniformly in $N\in\N$ and $0\leq s<t\leq 1$.
(Conditions $pq>2d$ and $\eta > pq-d$
are then fulfilled by any $q \in (\frac{4}{p}, \frac{2+\eta}{p})$, since $d=2$).
Since $Z^{\omega, \rm c}_{\beta_N, h_N}(0, \cdot)$ is defined on $[0,\infty)$ via linear interpolation, it suffices to prove \eqref{gamom2} for $s,t$ with $sN, tN\in\{0\}\cup\N$.
\medskip

\noindent
{\bf Step 2. Polynomial chaos expansion.}
To simplify notation, let us denote
$$
\Psi_{N,r} := Z^{\omega, \rm c}_{\beta_N, h_N}(0, r) =
\e\Bigg[ \prod_{n=1}^{r-1} e^{(\beta_N \omega_n - \Lambda(\beta_N) + h_N) \ind_{\{n\in \tau\}}} \, \Bigg|\, r\in\tau \Bigg] \qquad \mbox{for } r\in \N,
$$
and $\Psi_{N,0}:=1$. Since $e^{x \ind_{\{n\in \tau\}}} =
1 + (e^x-1) \ind_{\{n\in \tau\}}$ for all $x\in\R$, we set
\begin{equation}\label{eq:xi}
	\xi_{N,i}:= e^{\beta_N\omega_i -\Lambda(\beta_N) +h_N}-1,
\end{equation}
and rewrite $\Psi_{N,r}$ as a \emph{polynomial chaos expansion}:
\begin{equation}\label{PsiNr}
\Psi_{N,r} = \e\Bigg[ \prod_{i=1}^{r-1} (1+\xi_{N,i} \ind_{\{i\in\tau\}})
\,\Bigg|\, r\in \tau \Bigg] = \sum_{I\subset \{1,\ldots, r-1\}} \p(I\subset \tau | r\in\tau)
\prod_{i\in I}\xi_{N,i},
\end{equation}
using the notation $\{I\subset \tau\} := \bigcap_{i\in I} \{i\in\tau\}$.

Recalling \eqref{eq:scalingbetah>1} and \eqref{eq:omegacond}, it is easy to check that
\begin{equation}\label{eq:xiN}
\begin{aligned}
\bbE[\xi_{N,i}] & = e^{h_N}-1 = h_N + O(h_N^2),  \\
\sqrt{\bbvar(\xi_{N,i})} & =\sqrt{e^{2h_N}\big(e^{\Lambda(2\beta_N)-2\Lambda(\beta_N)}-1\big)}
= \sqrt{\beta_N^2 +O(\beta_N^3)} = \beta_N+ O(\beta_N^2),
\end{aligned}
\end{equation}
where we used the fact that $h_N = o(\beta_N)$ and we
Taylor expanded $\Lambda(t) := \log \bbE[e^{t\omega_1}]$, noting
that $\Lambda(0)= \Lambda'(0)= 0$ and $\Lambda''(0)= 1$.
Thus $h_N$ and $\beta_N$ are approximately the mean and standard deviation
of $\xi_{N,i}$.
Let us rewrite $\Psi_{N,r}$ in \eqref{PsiNr} using
normalized variables $\zeta_{N,i}$:
\begin{equation}\label{PsiNr2}
\Psi_{N,r} =  \sum_{I\subset \{1,\ldots, r-1\}} \psi_{N,r}(I) \prod_{i\in I}\zeta_{N,i},
\qquad \text{where} \qquad
\zeta_{N,i} := \frac{1}{\beta_N} \xi_{N,i} \,,
\end{equation}
where $\psi_{N,r}(\emptyset):=1$ and for $I=\{n_1 <n_2<\cdots <n_k\}\subset\N$,
recalling \eqref{un}, we can write
\begin{equation}\label{psiNrI}
\psi_{N,r}(I) = \psi_{N,r}(n_1,\ldots, n_k)
:= \beta_N^{|I|} \p(I\subset \tau | r\in\tau)
=  (\beta_N)^k \frac{1}{u(r)} \prod_{i=1}^{k+1} u(n_i - n_{i-1}),
\end{equation}
with $n_0 := 0$, $n_{k+1} := r$.

To prove \eqref{gamom2}, we write $Z^{\omega, \rm c}_{\beta_N, h_N}(0, sN)=\Psi_{N,q}$
and $Z^{\omega, \rm c}_{\beta_N, h_N}(0, tN)=\Psi_{N,r}$,
with $q := sN$ and $r := tN$,
so that $0\leq q<r\leq N$. For a given
\emph{truncation level} $m = m(q,r,N) \in (0, q)$, that we will later choose as
\begin{equation}\label{mchoice}
m = m(q,r,N) := \begin{cases}
0 &  \text{if } q\leq \sqrt{N(r-q)} \\
q-\sqrt{N(r-q)} &  \text{otherwise}
\end{cases} ,
\end{equation}
so that $0\leq m<q<r\leq N$, we write
$$
\Psi_{N,r} - \Psi_{N,q} = \Xi_1 +\Xi_2 -\Xi_3
$$
with
\begin{equation}\label{Xi}
\begin{gathered}
\Xi_1 \,=\, \sum_{I\subset \{1,\ldots, m\}}\big(\psi_{N,r}(I)-\psi_{N,q}(I)\big)
\prod_{i\in I} \zeta_{N,i}, \\
\Xi_2 \,=\, \sum_{I\subset \{1,\ldots, r-1\}\atop I\cap \{m+1, \ldots, r-1\}\neq \emptyset}
\psi_{N,r}(I)\prod_{i\in I} \zeta_{N,i},
\qquad \mbox{and} \qquad
\Xi_3 \,=\, \sum_{I\subset \{1,\ldots, q-1\}\atop I\cap \{m+1, \ldots, q-1\}\neq \emptyset}
\psi_{N,q}(I)\prod_{i\in I} \zeta_{N,i}.
\end{gathered}
\end{equation}
To establish \eqref{gamom2} and hence tightness, it suffices to show that for each $i=1,2,3$,
\begin{equation}\label{Ximom}
\exists C > 0, \ p \ge 1, \ \eta > 2: \qquad
\bbE[|\Xi_i|^p] \leq C \Big(\frac{r-q}{N}\Big)^\eta \qquad \forall
N\in\N, \ 0\leq q<r\leq N.
\end{equation}

\noindent
{\bf Step 3. Change of measure.}
We now estimate the moments of $\xi_{N,i}$
defined in \eqref{eq:xi}.
Since $(a+b)^{2k} \le 2^{2k-1}(a^{2k} + b^{2k})$, for all $k\in\N$,
and $h_N = O(\beta_N^2)$ by \eqref{eq:scalingbetah>1}, we can write
\begin{equation}\label{eq:2kbound}
\begin{split}
\bbE[\xi_{N,i}^{2k}] &\leq\, 2^{2k-1}e^{2k(h_N-\Lambda(\beta_N))}
\bbE\big[\big(e^{\beta_N\omega_i}-1\big)^{2k}\big] + 2^{2k-1}(e^{-\Lambda(\beta_N)+h_N}-1)^{2k} \\
&\leq\, C(k) \, \beta_N^{2k}  \, \bbE\Big[\Big(\frac{1}{\beta_N} \int_0^{\beta_N} \omega_i
e^{t\omega_i}\dd t\Big)^{2k}\Big] + O(\beta_N^{4k}+h_N^{2k}) \\
&\leq\, C(k)\beta_N^{2k-1} \int_0^{\beta_N} \bbE[\omega_i^{2k}e^{2k t\omega_i}]\dd t
+ o(\beta_N^{2k})\, =\,  O(\beta_N^{2k}),
\end{split}
\end{equation}
because $\bbE[\omega_i^{2k} e^{2kt\omega_i}]$ is uniformly bounded for $
t\in [0, t_0/4k]$ by our assumption \eqref{eq:omegacond}.

Recalling \eqref{eq:xiN}, \eqref{PsiNr2} and \eqref{eq:scalingbetah>1}, the random variables
$(\zeta_{N,i})_{i\in\N}$ are i.i.d.\ with
\begin{equation}\label{zetaNimom}
\bbE[\zeta_{N,i}] \,\underset{N\to\infty}{\sim}\,
\frac{\hh}{\hbeta} \frac{1}{\sqrt{N}} \,, \qquad
\bbvar[\zeta_{N,i}] \,\underset{N\to\infty}{\sim}\, 1,
\qquad \sup_{N,i\in\N}\bbE[(\zeta_{N,i})^{2k}] < \infty.
\end{equation}
It follows, in particular, that $\{\zeta_{N,i}^2\}_{i,N\in\N}$ are uniformly integrable.
We can then apply a change of measure result established in~\cite[Lemma B.1]{CSZ13},
which asserts that
we can construct i.i.d.\ random variables $(\tilde \zeta_{N,i})_{i\in\N}$
with marginal
distribution $\bbP(\tilde \zeta_{N,i}\in \dd x) = f_N(x) \bbP(\zeta_{N,i}\in \dd x)$,
for which there exists $C > 0$ such that
for all $p\in\R$ and $i, N\in\N$
\begin{equation}\label{tildezetamom}
\bbE[\tilde\zeta_{N,i}]=0, \qquad \bbE[\tilde \zeta_{N,i}^2] \leq 1 + C/\sqrt{N},
\qquad \text{and} \qquad \bbE[f_N(\zeta_{N,i})^p] \leq 1+ C/N.
\end{equation}
Let $\tilde \Xi_i$ be the analogue of $\Xi_i$ constructed from the $\tilde\zeta_{N,i}$'s
instead of the $\zeta_{N,i}$'s. By H\"older,
$$
\begin{aligned}
\bbE\big[|\Xi_i|^{l-1}\big] & = \bbE\Big[|\Xi_i|^{l-1} \prod_{i=1}^N f_N(\zeta_{N,i})^{\frac{l-1}{l}} \prod_{i=1}^N f_N(\zeta_{N,i})^{-\frac{l-1}{l}}\Big] \\
& \leq \bbE\big[|\tilde\Xi_i|^{l}\big]^{\frac{l-1}{l}} \bbE\big[f_N(\zeta_{N,1})^{1-l}\big]^{\frac{N}{l}}  \leq e^{\frac{C}{l}} \bbE\big[|\tilde\Xi_i|^{l}\big]^{\frac{l-1}{l}}.
\end{aligned}
$$
Relation \eqref{Ximom},
and hence the tightness of $\{Z^{\omega, \rm c}_{\beta_N, h_N}(\cdot, \cdot)\}_{N\in\N}$,
is thus reduced to showing
\begin{equation}\label{Ximom2}
\bbE\big[|\tilde\Xi_i|^{l}\big] \leq C \Big(\frac{r-q}{N}\Big)^\eta \qquad
\text{ for all  } N\in\N \text{ and } 0\leq q<r\leq N,
\end{equation}
for some $l\in\N$, $l\ge 2$ and $\eta>0$ satisfying $\eta > 2 \frac{l}{l-1}$.

\medskip
\noindent
{\bf Step 4. Bounding $\bbE[|\tilde\Xi_2|^{l}]$.}
We note that the bound for $\bbE[|\tilde\Xi_3|^{l}]$ is exactly the same
as that for $\bbE[|\tilde\Xi_2|^{l}]$, and hence will be omitted. First we write $\tilde \Xi_2$ as
\begin{equation}\label{tildeXi2(k)exp}
\tilde \Xi_2 = \sum_{k=1}^{r-1} \tilde \Xi_2^{(k)} , \qquad
\text{where} \qquad
\tilde \Xi_2^{(k)} :=
\sum_{|I|=k, I\subset \{1,\ldots, r-1\}\atop I\cap \{m+1, \ldots, r-1\}\neq \emptyset}
\psi_{N,r}(I)\prod_{i\in I} \tilde \zeta_{N,i},
\end{equation}
with $\tilde \Xi_2^{(k)}$ consisting of all terms of degree $k$.
The hypercontractivity established in \cite[Prop.~3.16 \& 3.12]{MOO10}
allows to estimates moments of order $l$ in terms
of moments of order $2$: more precisely, setting $\|X\|_p :=
\E[|X|^p]^{1/p}$, we have for all $l\geq 2$
\begin{equation}\label{tildeXi2(k)}
\Vert \tilde \Xi_2\Vert_l^l := \bbE[|\tilde\Xi_2|^{l}] \leq \Big(\sum_{k=1}^{r-1}
\Vert\tilde \Xi_2^{(k)}\Vert_l\Big)^l \leq \Big(\sum_{k=1}^{r-1} (c_l)^k
\Vert\tilde \Xi_2^{(k)}\Vert_2\Big)^l,
\end{equation}
where $c_l := 2\sqrt{l-1} \max_{N\in\N}\Big(
\frac{\Vert \tilde\zeta_{N,1}\Vert_l}{\Vert \tilde\zeta_{N,1}\Vert_2}\Big)$
is finite and depends only on $l$, by \eqref{zetaNimom}.

We now turn to the estimation of $\Vert\tilde \Xi_2^{(k)}\Vert_2$.
Let us recall the definition of $\psi_{N,r}$ in \eqref{psiNrI}.
It follows by
\eqref{tildezetamom} that $\bbvar(\tilde \zeta_{N,1})\leq 1+C/\sqrt{N}\leq 2$ for all $N$ large.
We then have
\begin{align}
\nonumber
& \Vert \tilde \Xi_2^{(k)}\Vert_2^2 = \bbE\big[\big(\tilde \Xi_2^{(k)}\big)^2\big] =
\sum_{y=0}^{k-1} \sum_{1\leq n_1<\cdots < n_y\leq  m \atop m+1\leq n_{y+1}<\cdots <n_k\leq r-1}
\psi^2_{N,r}(n_1,\ldots, n_k) \bbvar(\tilde\zeta_{N,1})^k \\
\nonumber
 & \leq
2^k \sum_{y=0}^{k-1} \sum_{1\leq n_1<\cdots <n_y\leq m \atop m+1\leq n_{y+1}<\cdots <n_k\leq r-1}
\frac{\beta_N^{2k}\, u(n_1)^2 u(n_2-n_1)^2\cdots u(r-n_k)^2}{u(r)^2}  \\
\label{tildeXi2cut}
& \le 4^k
\sum_{y=0}^{k-1} \!\idotsint\limits_{0<t_1<\cdots<t_y<\frac{m}{N} \atop
\frac{m}{N}<t_{y+1}<\cdots<t_{k}<\frac{r}{N}} \!\!
\frac{(\sqrt{N}\beta_N u(\lceil N t_1\rceil))^2 \cdots
(\sqrt{N}\beta_N u(r-\lceil N t_k\rceil))^2}{(\sqrt{N}\beta_N u(r))^2}\ \dd t_1\cdots \dd t_{k} \,.
\end{align}
It remains to estimate this integral, when $\frac{1}{2}<\alpha<1$ (the case $\alpha > 1$ is easy).
By \eqref{un}
\begin{equation*}
	\frac{1}{c} \frac{1}{L(\ell+1) (\ell+1)^{1-\alpha}}
	\le u(\ell) \le c \, \frac{1}{L(\ell+1) (\ell+1)^{1-\alpha}} \qquad
	\forall \ell\in\N \,,
\end{equation*}
for some $c \in (0,\infty)$. Since
$\lceil Nt\rceil-\lceil Ns\rceil + 1 \ge N(t-s)$, recalling \eqref{eq:scalingbetah>1} we obtain
\begin{equation*}
\sqrt{N}\beta_N u\big(\lceil Nt\rceil-\lceil Ns\rceil\big) \leq
c \frac{L(N)}{L\big(\lceil Nt\rceil-\lceil Ns\rceil+1\big)}
\frac{1}{(t - s)^{1-\alpha}} \,.
\end{equation*}
Let us now fix
\begin{equation}\label{alpha'}
\alpha':=
\begin{cases}
1 & \text{when  } \alpha>1 \\
\text{any number in } \big(\frac{1}{2}, \alpha\big) & \text{when } \frac{1}{2}<\alpha<1
\end{cases} .
\end{equation}
Since $L(\cdot)$ is slowly varying,
by Potter bounds \cite[Theorem 1.5.6]{BGT87}
for every $\epsilon > 0$ there is $D_\epsilon \in (0,\infty)$ such that
$L(a)/L(b) \le D_\epsilon \max\{ (a/b)^{\epsilon}, (b/a)^{\epsilon}\}$
for all $a,b\in\N$. It follows that
\begin{equation}\label{uNuniform}
\sqrt{N}\beta_N u\big(\lceil Nt\rceil-\lceil Ns\rceil\big) \leq
C \frac{1}{(t-s)^{1-\alpha'}},
\end{equation}
for some $C \in (0,\infty)$, uniformly in $0 < s<t\leq 1$ and $N\in\N$.
Analogously, again
by \eqref{un} and Potter bounds, if $0 \le s < t < \frac{r}{N}$ we have
\begin{equation}\label{uNuniform2}
\begin{split}
	& \frac{u(\lceil N t\rceil-\lceil N s\rceil)}{u(r)}
	\leq c^2 \frac{L(r+1)}{L(\lceil N t\rceil-\lceil N s\rceil + 1)}
	 \frac{r^{1-\alpha}}{(N(t - s))^{1-\alpha}}
	\le C \frac{(r/N)^{1-\alpha'}}{(t-s)^{1-\alpha'}} .
\end{split}
\end{equation}
Plugging \eqref{uNuniform} and \eqref{uNuniform2} into \eqref{tildeXi2cut},
and applying Lemma~\ref{L:integralbd}, then gives
\begin{align}
\nonumber
\Vert \tilde \Xi_2^{(k)}\Vert_2^2
&\leq C^k \sum_{y=0}^{k-1} \idotsint\limits_{0<t_1<\cdots<t_y<\frac{m}{N} \atop
\frac{m}{N}<t_{y+1}<\cdots<t_{k}<\frac{r}{N}}\frac{(r/N)^{2(1-\alpha')}}{t_1^{2(1-\alpha')}
(t_2- t_1)^{2(1-\alpha')}\cdots (r/N-t_k)^{2(1-\alpha')}}  \dd t_1\cdots \dd t_{k} \\
\nonumber
&\leq  C^k \sum_{y=0}^{k-1} C_1 e^{-C_2 k\log k} \Big(\frac{m}{N}\Big)^{(2\alpha'-1) y}
\Big(\frac{r-m}{N}\Big)^{(2\alpha'-1)(k-y)} \\
\label{tildeXi2}
&\leq  C_3 e^{-C_4k\log k} \Big(\frac{r-m}{N}\Big)^{2\alpha'-1},
\end{align}
where the last inequality follows by crude estimates
(observe that $m/N \le 1$).

We can substitute the bound \eqref{tildeXi2} into \eqref{tildeXi2(k)} to obtain
\begin{equation}\label{tildeXi2.2}
\bbE[|\tilde\Xi_2|^{l}] \leq \Big(\sum_{k=1}^{r}
(c_l)^k  (C_3)^{\frac{1}{2}} e^{-\frac{C_4}{2}k\log k}
\Big(\frac{r-m}{N}\Big)^{\alpha'-\frac{1}{2}} \Big)^l \leq
C \Big(\frac{r-m}{N}\Big)^{(\alpha'-\frac{1}{2}) l}
\end{equation}
for some $C$ depending only on $l$.
We now choose $m$ as in \eqref{mchoice}, so that
$$
\frac{r-m}{N}\leq \frac{r-q + \sqrt{N(r-q)}}{N}\leq 2\Big( \frac{r-q}{N}\Big)^{\frac{1}{2}},
$$
(if $m=0$ we first write $r=r-q+q$ and we use that in this case $q\leq \sqrt{N(r-q)}$),
hence
\begin{equation}\label{tildeXi2.3}
\bbE[|\tilde \Xi_2|^l] \leq C\, 2^{(\alpha'-\frac{1}{2}) l}
\Big(\frac{r-q}{N}\Big)^{(\alpha'-\frac{1}{2}) \frac{l}{2}}.
\end{equation}
Since $\alpha'>\frac{1}{2}$ by our choice in \eqref{alpha'}, relation
\eqref{Ximom2} is satisfied
with $\eta = (\alpha'-\frac{1}{2}) \frac{l}{2}$ (and
one has $\eta > 2 \frac{l}{l-1}$, as required, provided $l\in\N$ is chosen large enough).
\medskip

\noindent
{\bf Step 5: bounding $\bbE[|\tilde\Xi_1|^{l}]$.} Following the same steps as the bound for
$\bbE[|\tilde\Xi_2|^{l}]$, it suffices to establish an analogue of \eqref{tildeXi2} for
\begin{equation}\label{tildeXi1(k)}
\tilde \Xi_1^{(k)} := \sum_{1\leq n_1<\cdots < n_k\leq m} \big(\psi_{N,r}(n_1,\ldots, n_k)
-\psi_{N,q}(n_1,\ldots, n_k)\big) \prod_{i=1}^k \tilde \zeta_{N,n_i},
\end{equation}
where we recall that $0\leq m<q<r\leq N$,
because $m = m(q,r,N)$ is chosen as in \eqref{mchoice}.
If $m=0$ then $\tilde \Xi_1^{(k)} = 0$ and there is nothing to prove,
hence we assume $m>0$ henceforth.

Since $(\tilde \zeta_{N,i})_{i\in\N}$ are i.i.d.\ with
$\bbE[\tilde\zeta_{N,1}]=0$  and $\bbvar(\tilde \zeta_{N,1})\leq 1+C/\sqrt{N}\leq 2$
for $N$ large,
\begin{equation}
\Vert \tilde \Xi_1^{(k)}\Vert_2^2 = \bbE\big[\big(\tilde \Xi_1^{(k)}\big)^2\big] \leq
2^k \sum_{1\leq n_1<\cdots < n_k\leq m} \big(\psi_{N,r}(n_1,\ldots, n_k)-\psi_{N,q}
(n_1,\ldots, n_k)\big)^2 . \label{tildeXi1}
\end{equation}
Let $\epsilon$ be as in condition \eqref{renewalbd1}.
We first consider the case $r-q\geq \epsilon^2 r$, for which we bound
\begin{align}
	\label{tildeXi2.10}
	\Vert \tilde \Xi_1^{(k)}\Vert_2^2 & \leq  2^{k+1} \sum_{1\leq n_1<\cdots < n_k\leq m}
	\big(\psi_{N,r}(n_1,\ldots, n_k)^2 + \psi_{N,q}(n_1,\ldots, n_k)^2\big) \\
	& \leq 2^{k+1} \sum_{1\leq n_1<\cdots < n_k\leq r-1} \psi_{N,r}(n_1,\ldots, n_k)^2 +
	2^{k+1}\sum_{1\leq n_1<\cdots < n_k\leq q-1} \psi_{N,q}(n_1,\ldots, n_k)^2 \,.\nonumber
\end{align}
Applying the bound \eqref{tildeXi2} with $m=0$, since $q < r$, we obtain
\begin{equation}\label{tildeXi2.1}
	\Vert \tilde \Xi_1^{(k)}\Vert_2^2
	\leq 2^{k+1} C_3e^{-C_4k\log k}  \Big(\Big(\frac{r}{N}\Big)^{2\alpha'-1} +
	\Big(\frac{q}{N}\Big)^{2\alpha'-1}\Big) \leq C_5e^{-C_6 k\log k}
	\Big(\frac{r}{N}\Big)^{2\alpha'-1},
\end{equation}
By the same calculation as in \eqref{tildeXi2.2}, we then have, using $r-q\geq \epsilon^2 r$,
$$
\bbE[|\tilde\Xi_1|^{l}] \leq  C \Big(\frac{r}{N}\Big)^{(\alpha'-\frac{1}{2}) l}
\leq \frac{C}{\epsilon^{(2\alpha' - 1) l}}
\Big(\frac{r-q}{N}\Big)^{(\alpha'-\frac{1}{2}) l},
$$
which gives the desired bound \eqref{Ximom2}.

Now we consider the case $r-q\leq \epsilon^2 r$. Denote $I:=\{n_1, \ldots, n_k\}$.  Recalling the
definition of $\psi_{N,r}$ in \eqref{psiNrI}, we have
\begin{equation} \label{psiRogozin}
\big(\psi_{N,r}(I)-\psi_{N,q}(I)\big)^2 =
(\beta_N)^{2k} u(n_1)^2\cdots u(n_k-n_{k-1})^2
\bigg(\frac{u(r-n_k)}{u(r)}
- \frac{u(q-n_k)}{u(q)}\bigg)^2.
\end{equation}
Since we assume $m > 0$, by \eqref{mchoice} we have
 $m=q-\sqrt{N(r-q)}$ and $q, r\geq m+\sqrt{N}$.
Recalling that $u(n)=\p(n\in\tau)$, we can bound the last factor in \eqref{psiRogozin} as follows:
\begin{align*}
& \Big|\frac{\p(r-n_k\in\tau)}{\p(r\in\tau)} - \frac{\p(q-n_k\in\tau)}{\p(q\in\tau)}\Big|
= \Big|\frac{\p(q\in\tau) \p(r-n_k\in\tau) - \p(r\in\tau)\p(q-n_k\in\tau)}
{\p(q\in\tau)\p(r\in\tau)}\Big|\\
&= \bigg|\frac{\big[\p(q\in\tau)-\p(r\in\tau)\big]\p(r-n_k\in\tau) + \p(r\in\tau)
\big[\p(r-n_k\in\tau)-\p(q-n_k\in\tau)\big]}{\p(q\in\tau)\p(r\in\tau)}\bigg| \\
&\leq \frac{\big|\p(q\in\tau)-\p(r\in\tau)\big|\p(r-n_k\in\tau)}{\p(q\in\tau)\p(r\in\tau)}
+ \frac{\big|\p(q-n_k\in\tau)-\p(r-n_k\in\tau)\big|\p(q-n_k\in\tau)}
{\p(q-n_k\in\tau)\p(q\in\tau)}.
\end{align*}
We now apply \eqref{renewalbd1}, using the assumption
$r-q \leq \epsilon^2 r<\epsilon r$ and noting that
$$
\frac{(r-n_k) - (q-n_k)}{r-n_k} \leq \frac{r-q}{q-m} = \sqrt{\frac{r-q}{N}}
\leq \sqrt{\frac{r-q}{r}} \leq \epsilon ,
$$
which yields
\begin{align*}
	\bigg|\frac{u(r-n_k)}{u(r)}
	- \frac{u(q-n_k)}{u(q)}\bigg|
	&\leq C \Big(\frac{r-q}{r}\Big)^\delta \, \frac{u(r-n_k)}{u(r)}
	+ C \Big(\frac{r-q}{r-n_k}\Big)^\delta \, \frac{u(q-n_k)}{u(q)} \\
	&\leq C \Big(\frac{r-q}{r}\Big)^\delta \frac{u(r-n_k)}{u(r)}
	+ C \Big(\frac{r-q}{r}\Big)^{\delta/2} \frac{u(q-n_k)}{u(q)} .
\end{align*}
Plugging this into \eqref{psiRogozin} and recalling \eqref{psiNrI} then gives
$$
\begin{aligned}
\big(\psi_{N,r}(I)-\psi_{N,q}(I)\big)^2 & \leq 2C^2 \Big(\frac{r-q}{r}\Big)^{2\delta}
\psi_{N,r}(I)^2 + 2C^2 \Big(\frac{r-q}{r}\Big)^{\delta} \psi_{N,q}(I)^2 \\
& \leq 2C^2 \Big(\frac{r-q}{r}\Big)^{\delta} (\psi_{N,r}(I)^2+ \psi_{N,q}(I)^2).
\end{aligned}
$$
We can finally substitute this bound back into \eqref{tildeXi1} and follow the same calculations as
in \eqref{tildeXi2.10}-\eqref{tildeXi2.1},
with an extra factor $(\frac{r-q}{r})^{\delta}$, to obtain
$$
\Vert \tilde \Xi_1^{(k)}\Vert_2^2 \leq C_7 \, e^{-C_6 k\log k}
\Big(\frac{r-q}{r}\Big)^{\delta} \Big(\frac{r}{N}\Big)^{2\alpha'-1}
\leq C_7 e^{-C_6 k\log k} \Big(\frac{r-q}{N}\Big)^{\delta \wedge (2\alpha'-1)}.
$$
By the same calculation as in \eqref{tildeXi2.2}, we then have
$$
\bbE[|\tilde\Xi_1|^{l}] \leq  C \Big(\frac{r-q}{N}\Big)^{\frac{\delta \wedge (2\alpha'-1)}{2} l}.
$$
Since $\delta>0$ and $\alpha'>1/2$, this gives the desired bound
\eqref{Ximom2} for $\bbE[|\tilde\Xi_1|^{l}]$, provided $l\in\N$ is chosen large enough.
This completes the proof.\qed

\subsection{Proof of Theorem~\ref{T:bZprop}}
\label{S:proofbZprop}

We fix $\alpha\in (1/2, 1)$ and $T \in (0,\infty)$.
By Remark~\ref{R:Ken} and Lemma~\ref{lem:Bessel1} in the appendix,
we can construct a renewal process $\tau$ satisfying \eqref{eq:ass}, with $L(n)\to 1$
as $n\to\infty$, such that condition \eqref{renewalbd1} is satisfied.
By Theorem~\ref{T:cconv2}, for this particular renewal process,
the discrete partition functions
$\big(Z^{\omega,\rm c}_{\beta_N,h_N}(sN, tN)\big)_{0\leq s\leq t\leq T}$
converge in distribution as $N\to\infty$ to the continuum family
$\big(\bZ^{\alpha; W, \rm c}_{\hat\beta, \hat h}(s, t)\big)_{0\leq s\leq t\leq T}$,
viewed as random variables
in $C([0,T]^2_\le, \R)$. By Skorohod's representation theorem~\cite[Thm.~6.7]{B99}, we can couple
$\big(Z^{\omega,\rm c}_{\beta_N,h_N}\big)_{N\in\N}$
and $\bZ^{\alpha; W, \rm c}_{\hat\beta, \hat h}$
so that, a.s., $Z^{\omega,\rm c}_{\beta_N,h_N}(sN, tN)$ converges to
$\bZ^{\alpha; W, \rm c}_{\hat\beta, \hat h}(s, t)$ uniformly on $[0,T]^2_\le$.
We assume such a coupling from now on.

\smallskip

Property \eqref{it:2} is readily checked from the
Wiener chaos representation \eqref{cpinWC}. Alternatively,
one can observe that similar properties hold for the disordered
pinning partition functions $\big(Z^{\omega,\rm c}_{\beta_N,h_N}(i, j)\big)_{0\leq i\leq j}$,
which are preserved in the scaling limit.

\smallskip

We next prove \eqref{it:4}, where we may assume $0\leq s<u<t\leq T$. Let us fix a typical
realization of $\big(Z^{\omega,\rm c}_{\beta_N,h_N}\big)_{N\in\N}$ and
$\bZ^{\alpha; W, \rm c}_{\hat\beta, \hat h}$ under the above coupling. Let $a_N:= \lfloor sN\rfloor$,
$b_N:=\lfloor uN\rfloor$ and $c_N:=\lfloor tN\rfloor$. Recalling the definition of
$Z^{\omega, \rm c}_{\beta, h}$ in \eqref{eq:Zpinc}
and summing on the index $k\in\N$ for which
$\tau_k < b_N \le \tau_{k+1}$ and on the values
$i = \tau_k$, $j = \tau_{k+1}$, we obtain
\begin{equation}\label{Zdecomp}
\begin{aligned}
& Z^{\omega, \rm c}_{\beta_N, h_N}(a_N,c_N) \P(c_N-a_N\in \tau)  \\
=\ & \sum_{a_N\leq i<b_N} \sum_{b_N\leq j\leq c_N} Z^{\omega, \rm c}_{\beta_N, h_N}(a_N,i)
Z^{\omega, \rm c}_{\beta_N, h_N}(j,c_N)
  \ e^{(\beta_N \omega_i -\Lambda(\beta_N)+h_N)\ind_{\{i>a_N\}}}   \\
& \qquad \times \P(i-a_N\in \tau) \P(\tau_1=j-i) \P(c_N-j\in \tau) \
e^{(\beta_N \omega_j -\Lambda(\beta_N)+h_N)\ind_{\{j<c_N\}}}.
\end{aligned}
\end{equation}
Multiply both sides of \eqref{Zdecomp} by $N^{1-\alpha}$ and let $N\to\infty$. Since
$\P(n\in\tau)\sim \frac{C_\alpha}{n^{1-\alpha}}$ by \eqref{un},
$$
Z^{\omega, \rm c}_{\beta_N, h_N}(a_N,c_N)  N^{1-\alpha} \P(c_N-a_N\in \tau) \asto{N}
\frac{C_\alpha \, \bZ^{\alpha; W, \rm c}_{\hat \beta, \hat h}(s,t)}{(t-s)^{1-\alpha}} .
$$
For the RHS of \eqref{Zdecomp},
note that $(e^{(\beta_N \omega_i -\Lambda(\beta_N)+h_N)})_{0\leq i\leq TN}$
converge uniformly to $1$ as $N\to\infty$ (because
$\max\{ \omega_i: \, i \le TN\} = O(\log N)$ by Borel-Cantelli estimates,
cf.\ \eqref{eq:omegacond}).
Moreover, for $i=\lfloor xN\rfloor$
and $j=\lfloor yN\rfloor$, with $s<x<u<y<t$,
$$
Z^{\omega, \rm c}_{\beta_N, h_N}(a_N,\lfloor xN\rfloor)
Z^{\omega, \rm c}_{\beta_N, h_N}(\lfloor yN\rfloor,c_N)
\asto{N} \bZ^{\alpha; W,\rm c}_{\hat \beta, \hat h}(s,x)
\bZ^{\alpha; W,\rm c}_{\hat \beta, \hat h}(y,t) \qquad \mbox{uniformly},
$$
while by $\P(\tau_1=n) =
\frac{L(n)}{n^{1+\alpha}} $
and $\P(n\in\tau)\sim \frac{C_\alpha}{n^{1-\alpha}}$, cf. \eqref{eq:ass} and \eqref{un}, we get
\begin{equation}\label{tripuconv}
\begin{aligned}
& N^2 N^{1-\alpha} \P(\lfloor xN\rfloor-a_N\in \tau)
\P(\tau_1=\lfloor yN\rfloor-\lfloor xN\rfloor) \P(c_N-\lfloor yN\rfloor\in \tau) \\
& \qquad \asto{N} \ \frac{C_\alpha^2}{(x-s)^{1-\alpha} (y-x)^{1+\alpha} (t-y)^{1-\alpha}},
\end{aligned}
\end{equation}
for all $s < x < u < y < t$ (the convergence is even
uniform for $x-s, y-x, t-y\geq \epsilon$, for any $\epsilon >0$).
Again by \eqref{eq:ass} and \eqref{un}
with $L(n) \sim 1$, the LHS of \eqref{tripuconv}
is uniformly bounded by a constant multiple of the RHS, which is integrable
over $x\in (s,u)$ and $y\in (u,t)$.
Therefore, by a Riemann sum approximation, the RHS
of \eqref{Zdecomp}, multiplied by $N^{1-\alpha}$, converges to
$$
\int_{x\in (s,u)} \int_{y\in (u,t)}
\frac{C_\alpha\,\bZ^{\alpha; W, c}_{\hat\beta, \hat h}(s,x)}{(x-s)^{1-\alpha}} \,
\frac{1}{(y-x)^{1+\alpha}} \, \frac{C_\alpha\,
\bZ^{\alpha; W, c}_{\hat\beta, \hat h}(y,t)}{(t-y)^{1-\alpha}}\,
\dd x \, \dd y \,,
$$
which establishes \eqref{bUdecomp}.

\smallskip

We then turn to \eqref{it:1}, where we may restrict $ s,t \in [0,T]$.
The fact that
$\bZ^{\alpha; W, \rm c}_{\hat \beta, \hat h}(\cdot, \cdot)$
is a.s. continuous and non-negative follows readily from Theorem~\ref{T:cconv2}
(recall that $Z^{\omega,\rm c}_{\beta_N,h_N}(i, j) \ge 0$, cf. \eqref{eq:Zpinc}).
For the a.s.\ strict positivity,
we apply \eqref{bUdecomp} with $u := (s+t)/2$:
since $\bZ^{\alpha; W, \rm c}_{\hat \beta, \hat h} \ge 0$,
for any $\epsilon > 0$
restricting the integrals to $x \le s+\epsilon$
and $y \ge s-\epsilon$ yields the lower bound
\begin{equation} \label{eq:quaas}
\bZ^{\alpha; W, \rm c}_{\hat \beta, \hat h}(s,t) \geq
(t-s)^{1-\alpha} \!\!\!\!\!\!\!\!\!\!\!\!
\iint\limits_{x\in (s,(s+\epsilon)\wedge t), \; y\in ((t-\epsilon)\vee s,t)}
\frac{C_\alpha  \, \bZ^{\alpha; W, \rm c}_{\hat \beta, \hat h}(s,x)
\, \bZ^{\alpha; W, \rm c}_{\hat \beta, \hat h}(y,t)}
{(x-s)^{1-\alpha}(y-x)^{1+\alpha}(t-y)^{1-\alpha}}\, \dd x \, \dd y \, .
\end{equation}
Since $\bZ^{\alpha; W, \rm c}_{\hat \beta, \hat h}(u,u) =1$
for all $u \ge 0$, cf. \eqref{cpinWC},
by continuity a.s. there is (a random) $\epsilon>0$
such that $\bZ^{\alpha; W, \rm c}_{\hat \beta, \hat h}(u,v)>0$
for all $u,v \in [0,T]$ with $0 \le v-u \le \epsilon$. Observing that both
$s-x \le \epsilon$ and $y-t \le \epsilon$ in \eqref{eq:quaas} yields
that a.s. $\bZ^{\alpha; W, \rm c}_{\hat \beta, \hat h}(s,t) > 0$
for all $0 \le s \le t \le T$.

Lastly we prove \eqref{it:3}. For any $A>0$,
recalling \eqref{cpinWC} and setting
$\tilde W_t := A^{-1/2}W_{At}$, the change of variables
$t \mapsto u := t/A$ yields the equality in distribution
(jointly in $s,t$)
$$
\begin{aligned}
\bZ^{\alpha; W, \rm c}_{\hat\beta, \hat h}(As, At) & = 1 + \sum_{k=1}^\infty
\ \ \idotsint\limits_{As<t_1<\cdots <t_k<At} \bpsi^{\alpha; \rm c}_{As,At}(t_1,\ldots, t_k)
\prod_{i=1}^k  (\hat\beta\, \dd W_{t_i} + \hat h\, \dd t_i) \\
& \overset{\mathrm{dist.}}{=} 1 + \sum_{k=1}^\infty \ \ \idotsint\limits_{s<u_1<\cdots < u_k<t}
\bpsi^{\alpha; \rm c}_{As,At}(A u_1,\ldots, A u_k) \prod_{i=1}^k
(A^{1/2} \hat\beta\, \dd \tilde W_{u_i} + A\hat h\,\dd u_i) .
\end{aligned}
$$
Since $\bpsi^{\alpha; \rm c}_{As,At}(A u_1,\ldots, A u_k)
= A^{(\alpha-1) k} \bpsi^{\alpha; \rm c}_{s,t}(u_1,\ldots, u_k)$, by \eqref{cpinWCphi},
it follows that
\begin{equation*}
	\bZ^{\alpha; W, \rm c}_{\hat\beta, \hat h}(As, At)
	\overset{\mathrm{dist.}}{=}
	\bZ^{\alpha; \tilde W, \rm c}_{A^{\alpha-1/2}\hat\beta, A^\alpha\hat h}(s, t) \,.
\end{equation*}
Since $\tilde W = (\tilde W_t)_{t\ge 0}$ is still a standard Brownian motion, the proof is completed.
\qed

\section{Characterization and Universality of the CDPM}
\label{S:CDPMchar}

In this section we prove Theorems~\ref{th:main1} and~\ref{T:cconv3}.
We recall that $\cC$ is the space of all closed subsets of $\R$,
and refer to Appendix~\ref{A:RCS} for some key facts on $\cC$-valued random variables
(in particular for the notion of \emph{restricted f.d.d.}, cf.\ \S\ref{sec:restricted}).
Let us summarize our setting:
\begin{itemize}
\item we have two independent sources
of randomness: a renewal process $\tau = (\tau_n)_{n\ge 0}$
satisfying \eqref{eq:ass} and \eqref{eq:renass0+}, and
an i.i.d.\ sequence $\omega = (\omega_n)_{n\ge 1}$
satisfying \eqref{eq:omegacond};

\item we fix $T > 0$
and consider the \emph{conditioned pinning model} $\p^{\omega, \rm c}_{NT,\beta,h}$,
defined in \eqref{eq:condpm} and \eqref{eq:Gweight},
with the parameters $\beta = \beta_N$ and $h = h_N$
chosen as in \eqref{eq:scalingbetah}.
\end{itemize}
Let us denote by $X_N$ the rescaled set $\tau/N \cap [0,T]$,
cf.\ \eqref{eq:rescN}, under the law $\p^{\omega, \rm c}_{NT, \beta_N ,h_N}$.
If we \emph{fix} a realization of $\omega$, then $X_N$ is a
$\cC$-valued random variable (with respect to $\tau$).

\smallskip

Our strategy to prove Theorems~\ref{th:main1} and~\ref{T:cconv3} is based on two main steps:
\begin{enumerate}
\item\label{it:cou1}
first we define a suitable coupling of $\omega$ with a standard Brownian motion $W$;

\item\label{it:cou2}  then we show that, for $\bbP$-a.e.\ \emph{fixed}
realization of $(\omega, W)$,
the restricted f.d.d.\ of $X_N$ converge weakly as $N\to\infty$
to those given in the right hand side of \eqref{CDPMfdd}.

\end{enumerate}
We can then apply
Proposition~\ref{P:cK+char2}~\eqref{it:restiii}, which guarantees that the densities in
\eqref{CDPMfdd} are the restricted f.d.d.
of a $\cC$-valued random variable $X_\infty$, whose law on $\cC$
we denote by $\bP^{\alpha; W, \rm c}_{T, \hat \beta, \hat h}$;
furthermore, $X_N$ converges in distribution on $\cC$ toward $X_\infty$
as $N\to\infty$, for $\bbP$-a.e.\ \emph{fixed}
realization of $(\omega, W)$.
This is nothing but Theorem~\ref{th:main1} in a strengthened form,
with a.s. convergence instead of convergence in distribution
(thanks to the coupling).
Theorem~\ref{T:cconv3} is also proved, once we note that
$\bP^{\alpha; W, \rm c}_{T, \hat \beta, \hat h}$
is the \emph{unique} probability law on $\cC$ satisfying
conditions \eqref{it:1b} and \eqref{it:2b} therein,
because restricted f.d.d. characterize laws on $\cC$,
cf.\ Proposition~\ref{P:cK+char2}~\eqref{it:resti}.
It only remains to prove points \eqref{it:cou1} and \eqref{it:cou2}.

\medskip

By Theorem~\ref{T:cconv2}, the family
$Z_N := \big(Z^{\omega,\rm c}_{\beta_N,h_N}(sN, tN)\big)_{0\leq s\leq t\leq T}$
of discrete partition functions defined in \eqref{eq:Zpinc},
viewed as a $C([0,T]^2_\le, \R)$-valued random variable,
converges in distribution to the continuum family
$\bZ := \big(\bZ^{\alpha; W, \rm c}_{\hat\beta, \hat h}(s, t)\big)_{0\leq s\leq t\leq T}$
as $N\to\infty$.
Note that $Z_N$ is a function of
$\omega_{(0,N)} := (\omega_1, \ldots, \omega_{N-1})$,
while $\bZ$ is a function of a standard Brownian motion $W = (W_t)_{t\ge 0}$.
By an extension of Skorohod's representation
theorem~\cite[Cor.~5.12]{K97},
we can couple the discrete environments $(\omega_{(0,N)})_{N\in\N}$
and $W$ on the same probability space, so that $Z_N \to \bZ$ a.s..
This completes point \eqref{it:cou1}.

\medskip

Coming to point \eqref{it:cou2}, we prove the convergence of
the restricted f.d.d. of $X_N$ ,
i.e. the laws of the vectors $(\tg_{t_1}(X_N), \ldots, \tg_{t_k}(X_N))$
restricted on the event $A_{t_1,\ldots, t_k}^{X_N}$ defined in
\eqref{eq:restricted}.
Since $X_N = \tau/N \cap [0,T]$ under
the pinning law $\p_{TN,\beta_N,h_N}^{\omega, \rm c}$,
we fix $k \in \N$ and $0 < t_1 < \ldots < t_k < T$,
as well as a continuous and bounded function $F: \R^{2k} \to \R$,
and we have to show that
\begin{equation}\label{eq:toproJ}
\begin{split}
	I_N := \E_{TN,\beta_N,h_N}^{\omega, \rm c} \Big[ F\big(\tg_{t_1}(\tau/N), &
	\, \td_{t_1}(\tau/N), \ldots, \tg_{t_k}(\tau/N), \td_{t_k}(\tau/N)\big)
	\ind_{A_{t_1,\ldots, t_k}^{\tau/N}} \Big]
\end{split}
\end{equation}
converges as $N\to\infty$ to the integral of $F$
with respect to the density in \eqref{CDPMfdd}, i.e.
\begin{equation} \label{eq:intJ}
\begin{split}
	I := \idotsint\limits_{0<x_1<t_1<y_1<x_2<t_2\atop <\cdots <x_k<t_k<y_k<T}
	& F(x_1,y_1,\ldots, x_k, y_k) \times
	\Bigg\{\, \frac{\prod_{i=0}^k \bZ^{\alpha; W, \rm c}_{\hat \beta, \hat h}(y_{i},x_{i+1})}
	{\bZ^{\alpha; W, \rm c}_{\hat \beta, \hat h}(0,T)} \Bigg\}\,
	\\
	& \times \Bigg[ \prod_{i=1}^{k}
	\frac{C_\alpha}{(x_{i} - y_{i-1})^{1-\alpha}
	(y_i-x_i)^{1+\alpha}} \Bigg] \frac{T^{1-\alpha}}{(T-y_k)^{1-\alpha}} \, \dd x \, \dd y \,,
\end{split}
\end{equation}
where we set $y_0 := 0$, $x_{k+1} := T$ and
$\dd x \, \dd y$ is a shorthand for $\dd x_1 \dd y_1 \cdots \dd x_k \dd y_k$.

Denoting by $\P_{N}^{\rm c}$ the law of the
renewal process $\tau \cap [0,N]$
conditioned to visit $N$,
\begin{equation}\label{eq:PN}
	\P_N^{\rm c}(\,\cdot\,) := \P( \tau \cap [0,N] \in \,\cdot\,|\, N \in \tau) \,,
\end{equation}
the  pinning law $\p^{\omega, \rm c}_{N,\beta,h}$
can be written as follows,
cf.\ \eqref{eq:Gweight}, \eqref{eq:condpm} and \eqref{eq:Zpinc}:
\begin{equation}\label{eq:Gcweight}
	\frac{\p^{\omega, \rm c}_{N,\beta,h}(\tau)}
	{\p_N^{\rm c}(\tau)} :=
	\frac{1}{Z^{\omega, \rm c}_{\beta,h}(0,N)} e^{\sum_{n=1}^{N-1} (\beta \omega_n
	-\Lambda(\beta) + h) \ind_{\{n\in\tau\}}} \, .
\end{equation}
In particular, the law
$\p^{\omega, \rm c}_{TN,\beta,h}$ reduces to
$\P_{TN}^{\rm c}$ for $\beta = h = 0$. In this special case, the convergence $I_N \to I$
is shown in the proof of Proposition~\ref{P:univrenew},
cf.\ \eqref{eq:topro} and the following lines,
exploiting the renewal decomposition \eqref{eq:Inbasic}
for $I_N$. In the general case,
with $\p^{\omega, \rm c}_{TN,\beta,h}$ instead
$\P_{TN}^{\rm c}$, we have a completely analogous decomposition, thanks to \eqref{eq:Gcweight}:
\begin{equation*}
\begin{split}
	I_N = \frac{1}{N^{2k}} \sum_{0 \le a_1 \le Nt_1} \,
	& \sum_{N t_1 < b_1 \le a_2 \le Nt_2}
	\cdots \sum_{N t_{k-1} < b_{k-1} \le a_k \le Nt_k} \, \sum_{N t_k < b_k \le NT}
	F\Big(\frac{a_1}{N}, \frac{b_1}{N}, \ldots, \frac{a_k}{N}, \frac{b_k}{N}\Big) \\
	& \times
	\Bigg\{ \prod_{i=1}^k e^{\beta_N (\omega_{a_i} + \omega_{b_i}) - 2 \Lambda(\beta_N)
	+ 2 h_N}\Bigg\}
	\Bigg\{\, \frac{\prod_{i=0}^k Z^{\omega, \rm c}_{\beta_N, h_N}(b_{i},a_{i+1})}
	{Z^{\omega, \rm c}_{\beta_N, h_N}(0,NT)} \Bigg\}\, \\
	& \times \Bigg[ \prod_{i=1}^{k}
	N^2 \, u(a_i - b_{i-1}) K(b_i - a_i) \Bigg] \frac{u(\lfloor NT\rfloor-b_k)}{u(\lfloor NT\rfloor)} .
\end{split}
\end{equation*}
We stress that the difference with respect to \eqref{eq:Inbasic}
is only given by the two terms in brackets appearing in the middle line.
The first term in brackets converges to $1$ as $N\to\infty$,
because $\max_{0 \le n \le NT} |\omega_n| = O(\log N)$
(as we already remarked in \S\ref{S:proofbZprop}).
When we set $a_i = N x_i$ and $b_i = N y_i$,
the second term in brackets
converges to its analogue in \eqref{eq:intJ} involving the continuum partition functions,
for $\bbP$-a.e. fixed realization of $(\omega,W)$ (thanks to our coupling),
and is uniformly bounded by some (random) constant,
because the continuum partition functions are a.s. continuous and strictly positive,
cf.\ Theorem~\ref{T:bZprop}~\eqref{it:1}.
Since the convergence of \eqref{eq:topro} to \eqref{eq:intI} is shown
by a Riemann sum approximation, the convergence
of \eqref{eq:toproJ} to \eqref{eq:intJ} follows immediately,
completing the proof of point \eqref{it:cou2}.
\qed

\section{Key properties of the CDPM}\label{S:CDPMprop}

In this section, we prove Theorems~\ref{T:averageabs} and~\ref{T:quenchsing}.
The parameters $\alpha \in (\frac{1}{2}, 1)$, $T > 0$, $\hbeta > 0$ and $\hh \in \R$
are fixed throughout the section.
We use in an essential way the continuum partition functions
$\big(\bZ^{\alpha; W, \rm c}_{\hat\beta, \hat h}(s, t)\big)_{0\leq s\leq t\leq T}$,
cf.\ Theorems~\ref{T:cconv2} and~\ref{T:bZprop},
and the characterization
of the CDPM quenched law $\bP^{\alpha; W, \rm c}_{T, \hat \beta, \hh}$
in terms of restricted f.d.d., given in Theorem~\ref{T:cconv3}.

\medskip

\noindent
{\bf Proof of Theorem~\ref{T:averageabs}.}
Assume that $\hh = 0$,
and recall definition \eqref{eq:bPnT} of the ``reference law'' $\bP^{\alpha; \rm c}_T$.
We will show at the end of the proof the following equality of two probability measures on $\cC$:
\begin{equation}\label{aveqlaw}
\bbE\big[
\bZ^{\alpha; W, \rm c}_{\hat\beta,0}(0,T) \, \bP^{\alpha; W, \rm c}_{T, \hat \beta, 0} (\cdot)
\big] = \bP^{\alpha; \rm c}_T(\cdot).
\end{equation}
Let us assume this for the moment.

By \eqref{aveqlaw},
if $\bP^{\alpha; \rm c}_T(A) = 0$ for some $A \subseteq \cC$, then
$\bP^{\alpha; W, \rm c}_{T, \hat \beta, 0} (A) = 0$ for $\bbP$-a.e.
$W$ (because  $\bZ^{\alpha; W, \rm c}_{\hat\beta,0}(0,T)>0$ a.s., by
Theorem~\ref{T:bZprop}~\eqref{it:1}), hence
$\bbE[\bP^{\alpha; W, \rm c}_{T, \hat \beta, 0} (A)] = 0$.
This shows that the law $\bbE[\bP^{\alpha; W, \rm c}_{T, \hat \beta, 0}(\cdot)]$
is absolutely continuous with respect to
$\bP^{\alpha; \rm c}_T(\cdot)$,
proving Theorem~\ref{T:averageabs} for $\hh = 0$.

\smallskip

We now turn to the case $\hh \ne 0$. By
Remark~\ref{R:CM}, the continuum partition functions
$\big(\bZ^{\alpha; W, \rm c}_{\hat\beta, \hat h}(s, t)\big)_{0\leq s\leq t\leq T}$
have a law that, for $\hh \ne 0$,
is absolutely continuous with respect to case $\hh = 0$,
with Radon-Nikodym density $\mathfrak{f}_{T, \hbeta, \hh}(W)$
given in \eqref{eq:RN}.
Since the restricted f.d.d. of $\bP^{\alpha; W, \rm c}_{T, \hat \beta, \hat h}$ are
expressed in terms of continuum partition functions,
cf.\ \eqref{CDPMfdd}, the two probability measures
$\bbE[\bP^{\alpha; W, \rm c}_{T, \hat \beta, \hh}(\cdot)]$ and
$\bbE[ \mathfrak{f}_{T, \hbeta, \hh}(W) \, \bP^{\alpha; W, \rm c}_{T, \hat \beta, 0}(\cdot)]$
on $\cC$
have the same restricted f.d.d. and hence are identical,
by Proposition~\ref{P:cK+char2}~\eqref{it:resti}.
As a consequence, if $\bbE[\bP^{\alpha; W, \rm c}_{T, \hat \beta, \hh} (A)] = 0$
for $\hh = 0$, the same is true also for $\hh \ne 0$, completing
the proof of Theorem~\ref{T:averageabs}.

\smallskip

It remains to establish \eqref{aveqlaw}.
Note that its LHS is indeed a probability law on $\cC$,
since $\bbE[\bZ^{\alpha; W, \rm c}_{\hat\beta,0}(0,T)]=1$ by the Wiener-chaos expansion
in \eqref{cpinWC} with $\hat h=0$.
It suffices to show that the LHS and RHS in \eqref{aveqlaw}
have the same restricted f.d.d., by Proposition~\ref{P:cK+char2}~\eqref{it:resti},
and this follows immediately from
relations \eqref{fddalpha} and \eqref{CDPMfdd}, because
$\bbE[\bZ^{\alpha; W, \rm c}_{\hat \beta, 0}(y_{i},x_{i+1})] = 1$.

\smallskip

Lastly, we note that the $\alpha$-stable regenerative set $\btau$ a.s.\ has Hausdorff dimension
$\alpha$ (see e.g.~\cite[Thm.~III.15]{B96}), and the same holds a.s.\ under the conditioned
measure $\bP^{\alpha; \rm c}_T$, which then carries over to the quenched law
$\bP^{\alpha; W, \rm c}_{T, \hbeta, \hh}$ of the CDPM,
for a.e.\ realization of $W$.
\qed

\medskip

\noindent
{\bf Proof of Theorem~\ref{T:quenchsing}.}
Let us set
$$
\cC_{0,T} := \{K\in \cC: 0,T \in K \subset [0,T]\}.
$$
By construction, $\bP^{\alpha; W, \rm c}_{T, \hbeta, 0}$ and $\bP^{\alpha; \rm c}_T$
are probability measures on $\cC_{0,T}$, equipped with the Borel $\sigma$-algebra $\cF$.
Recalling the definition \eqref{eq:gtdt0} of the maps $\tg_t, \td_t$,
for $n\in\N$ let $\cF_n$ be the
$\sigma$-algebra on $\cC_{0,T}$ generated by $\tg_{\frac{i}{2^n}T}$ and $\td_{\frac{i}{2^n}T}$ for
$1\leq i\leq 2^n-1$. Then $(\cF_n)_{n\in\N}$ is a filtration on $\cC_{0,T}$ that generates the
Borel $\sigma$-algebra $\cF$ on $\cC_{0,T}$, by Lemma~\ref{th:basic}.

Let $f^W_n: \cC_{0,T} \to [0,\infty]$ be the Radon-Nikodym derivative of
$\bP^{\alpha; W, \rm c}_{T, \hbeta, \hh}$ with respect to $\bP^{\alpha; \rm c}_T$ on
$(\cC_{0,T}, \cF_n)$.
If $\btau$ is a $\cC_{0,T}$-valued random variable with law
$\bP^{\alpha; \rm c}_T$, then $(f^W_n(\btau))_{n\in \N}$ is a non-negative martingale adapted to
the filtration $(\cF_n)_{n\in\N}$, and $\bP^{\alpha; W, \rm c}_{T, \hbeta, \hh}$ is singular w.r.t.\
$\bP^{\alpha; \rm c}_T$ if and only if $f^W_n(\btau) \to 0$ a.s.. To prove
Theorem~\ref{T:quenchsing}, it suffices to show
that, under the joint law of $\btau$ and $W$,
\begin{equation}\label{fWnlim}
f^W_n(\btau) \xrightarrow[n\to\infty]{} 0 \quad \ \ \text{in probability} \,,
\end{equation}
because we already know that the martingale limit $\lim_{n\to\infty} f^W_n(\btau)$ exists a.s..

We next identify $f^W_n(\btau)$. Without loss of generality, assume $T=1$.
To remove duplicates among the
random variables $\tg_{\frac{i}{2^{n}}}$, $\td_{\frac{i}{2^{n}}}$, for $1\leq i\leq 2^n-1$, let us set
\begin{gather*}
	I_n(\btau) := \{ 2\leq i\leq 2^n - 1:
\ \btau \cap (\tfrac{i-1}{2^{n}}, \tfrac{i}{2^{n}}] \neq \emptyset\} \,,\\
	a_{n,j}(\btau) := \min \big\{ \btau \cap [\tfrac{j-1}{2^{n}}, \tfrac{j}{2^{n}}] \big\} \,,
	\quad \
	b_{n,j}(\btau) := \max \big\{ \btau \cap [\tfrac{j-1}{2^{n}}, \tfrac{j}{2^{n}}] \big\} \,.
\end{gather*}
Then we claim that the following explicit expression for $f_n^W(\btau)$ holds:
\begin{equation}\label{fWnbtau}
	f_n^W(\btau) = \frac{\prod_{j \in \{1\} \cup I_n(\btau) \cup \{2^n\}}
	\bZ^{\alpha; W, \rm c}_{\hbeta,\hh}(a_{n,j}(\btau), b_{n,j}(\btau))}
	{\bZ^{\alpha; W, \rm c}_{\hbeta,\hh}(0,1)} \,.
\end{equation}
In order to prove it, first note that $f_n^W(\btau)$
must necessarily be a function of the vector
\begin{equation} \label{eq:vecto}
	V_n(\btau) := \Big( \tg_{\frac{1}{2^n}}(\btau), \,
	\big(\td_{\frac{j-1}{2^n}}(\btau) , \tg_{\frac{j}{2^n}}(\btau)\big)_{j \in I_n(\btau)}, \,
	\td_{1-\frac{1}{2^n}}(\btau) \Big) \,,
\end{equation}
because the variables $\tg_{\frac{i}{2^{n}}}(\btau)$, $\td_{\frac{i}{2^{n}}}(\btau)$
for $i\not\in I_n(\btau)$ are just repetitions of those
in $V_n(\btau)$.\footnote{The presence of $\tg_{\frac{1}{2^n}}(\btau)$
and $\td_{1-\frac{1}{2^n}}(\btau)$ in \eqref{eq:vecto}
is due to the fact that $0$ and $T$ are accumulation points of $\btau$,
$\bP^{\alpha; \rm c}_1(\dd\btau)$-a.s., hence $\btau \cap (0,\frac{1}{2^n}] \ne \emptyset$
and $\btau \cap (1-\frac{1}{2^n}, 1] \ne \emptyset$.}
Then observe that $V_n(\btau)$ can be rewritten equivalently as
\begin{equation*}
	V_n(\btau) = \Big( \tg_{\frac{1}{2^n}}(\btau), \,
	\td_{\frac{1}{2^n}}(\btau), \,
	\big(\tg_{\frac{j}{2^n}}(\btau), \td_{\frac{j}{2^n}}(\btau)\big)_{j \in I_n(\btau)} \Big) \,,
\end{equation*}
because if $j < j'$ are consecutive points in $I_n(\btau)$ then
$\td_{\frac{j}{2^n}}(\btau) = \td_{\frac{j'-1}{2^n}}(\btau)$.
Finally, note that on the event $I_n(\btau) = J$,
defining $t_1 := \frac{1}{2^n}$ and $\{t_2, \ldots, t_k\} := \{\frac{j}{2^n}\}_{j\in J}$,
the density of the random vector $V_n(\btau)$
(w.r.t.\ Lebesgue measure on $\R^{2k}$)
under $\bP^{\alpha; \rm c}_1$ is given by \eqref{fddalpha},
while under $\bP^{\alpha; W, \rm c}_{1, \hbeta, \hh}$ is given by \eqref{CDPMfdd};
then \eqref{fWnbtau} follows comparing \eqref{fddalpha} and \eqref{CDPMfdd}.

\smallskip

Since $\bZ^{\alpha; W, \rm c}_{\hbeta,0}(0,1)>0$ a.s.\ by Theorem~\ref{T:bZprop}, relation
\eqref{fWnlim} will follow by showing that
\begin{equation}\label{fWnlim2}
\mbox{$\bP^{\alpha; \rm c}_1$-a.s.}, \qquad
\bbE\big[\big(f^W_n(\btau) \bZ^{\alpha; W, \rm c}_{\hbeta, \hh}(0,1)\big)^\gamma\big]
\asto{n} 0 \qquad \mbox{for some } \gamma\in (0,1).
\end{equation}
We can reduce to the case $\hh = 0$
using Remark~\ref{R:CM}, cf.\ in particular \eqref{eq:RN}, writing
\begin{equation*}
\begin{split}
	\bbE\big[\big(f^W_n(\btau) \bZ^{\alpha; W, \rm c}_{\hbeta, \hh}(0,1)\big)^\gamma\big]
	& = \bbE\big[
	\mathfrak{f}_{1, \hbeta, \hh}(W)
	\big(f^W_n(\btau) \bZ^{\alpha; W, \rm c}_{\hbeta, 0}(0,1)\big)^\gamma
	 \big]  \\
	& \le \bbE\big[
	\mathfrak{f}_{1, \hbeta, \hh}(W)^{\frac{p}{p-1}}\big]^{\frac{p-1}{p}}
	\, \bbE\big[
	\big(f^W_n(\btau) \bZ^{\alpha; W, \rm c}_{\hbeta, 0}(0,1)
	\big)^{\gamma p}
	 \big]^{1/p} \,.
\end{split}
\end{equation*}
Choosing $p \in (1,\infty)$ close to $1$, it suffices to prove
\eqref{fWnlim2} in the special case $\hh = 0$.

\smallskip

Henceforth we fix $\hh = 0$.
For a given realization of $\btau$, we note that the factors in the numerator of \eqref{fWnbtau}
are independent, by Theorem~\ref{T:bZprop}~\eqref{it:1}. Therefore
\begin{equation}\label{fWnlim3}
\begin{aligned}
\bbE\big[\big(f^W_n(\btau) \bZ^{\alpha; W, \rm c}_{\hbeta,0}(0,1)\big)^\gamma\big]
& = \prod_{j\in \{1\} \cup I_n(\btau) \cup \{2^n\}} \bbE\big[\big( \bZ^{\alpha; W, \rm c}_{\hbeta,0}
(a_{n,j}(\btau), b_{n,j}(\btau)) \big)^\gamma\big] \\
& = \prod_{j\in \{1\} \cup I_n(\btau) \cup \{2^n\}}
\bbE\big[\big( \bZ^{\alpha; W, \rm c}_{\hbeta_{n,j}(\btau),0}(0,1) \big)^\gamma\big],
\end{aligned}
\end{equation}
where we set $\hbeta_{n,j}(\btau) :=
\hbeta (b_{n,j}(\btau)-a_{n,j}(\btau))^{\alpha-1/2}$ and we used the translation invariance and
scaling property of $\bZ^{\alpha; W, \rm c}_{\hbeta,0}(\cdot, \cdot)$
established in Theorem~\ref{T:bZprop}~\eqref{it:2}-\eqref{it:3}.

We claim that for any $\gamma \in (0,\frac{1}{2})$
there exists $c=c(\gamma)>0$ such that for $\hbeta>0$ sufficiently small,
\begin{equation}\label{Zfracmom}
\bbE\big[\big( \bZ^{\alpha; W, \rm c}_{\hbeta,0}(0,1) \big)^\gamma\big] \leq 1-c \hbeta^2
\leq e^{-c\hbeta^2}.
\end{equation}
Substituting this bound into \eqref{fWnlim3} then gives
$$
\log \bbE\big[\big(f^W_n(\btau) \bZ^{\alpha; W, \rm c}_{\hbeta,0}(0,1)\big)^\gamma\big]
\leq  - c\hbeta^2 \!\!\!\!\sum_{j\in \{1\} \cup I_n(\btau) \cup \{2^n\}}\!\!\!\!
(b_{n,j}(\btau)-a_{n,j}(\btau))^{2\alpha-1} \,.
$$
The RHS diverges $\bP^{\alpha; \rm c}_1(\dd\btau)$-a.s.\ as $n\to\infty$,
because
$\{[a_{n,j}(\btau), b_{n,j}(\btau)]\}_{j\in \{1\} \cup I_n\cup\{2^{n}\}}$ is a covering of $\btau$
with balls of
diameter at most $2^{-n}$, and $\btau$ a.s.\ has Hausdorff dimension $\alpha$, which is strictly
larger than $2\alpha-1$ for $\alpha \in (\frac{1}{2},1)$.
The divergence follows from the definition of the
Hausdorff dimension (see e.g.~\cite[Section~III.5]{B96}.

Lastly we prove \eqref{Zfracmom}. By \eqref{cpinWC}, we have the representation
$$
\bZ^{\alpha; W, \rm c}_{\hbeta, 0}(0,1) = 1 + \sum_{k=1}^\infty \hbeta^k Y_k \,,
$$
where $Y_k$ is a random variable in the $k$-th order Wiener chaos expansion,
and we recall that the series converges in $L^2$ for all $\hbeta>0$.
By Taylor expansion, there exist $\epsilon, C>0$ such that
$$
(1+x)^\gamma \leq 1+ \gamma x - Cx^2 \qquad \mbox{for all } |x|\leq \epsilon.
$$
For later convenience, let us define
\begin{equation*}
	\cS_\hbeta := \sum_{k=1}^\infty \hbeta^k Y_k \,, \qquad
	\cT_\hbeta := \sum_{k=2}^\infty \hbeta^k Y_k \,.
\end{equation*}
We then obtain
\begin{align}
\nonumber
\bbE\big[\big( & \bZ^{\alpha; W, \rm c}_{\hbeta,0}  (0,1) \big)^\gamma\big] =
\bbE\big[(1+\cS_\hbeta)^\gamma\ind_{\{|\cS_\hbeta|\leq \epsilon\}}\big]
+ \bbE\big[(1+\cS_\hbeta)^\gamma
\ind_{\{|\cS_\hbeta|> \epsilon\}}\big] \\
\nonumber
& \leq 1+ \gamma\bbE[\cS_\hbeta\ind_{\{|\cS_\hbeta|\leq \epsilon\}}] -
C\bbE[\cS_\hbeta^2\ind_{\{|\cS_\hbeta|\leq \epsilon\}}] + \bbE[1+\cS_\hbeta]^\gamma
\bbP(|\cS_\hbeta|\geq \epsilon)^{1-\gamma}  \\
\label{eq:trelin}
& = 1 - C \bbE[\cS_\hbeta^2]
+ \big\{ - \gamma \bbE[\cS_\hbeta\ind_{\{|\cS_\hbeta| > \epsilon\}}]
+ C\bbE[\cS_\hbeta^2\ind_{\{|\cS_\hbeta| > \epsilon\}}] +
\bbP(|\cS_\hbeta|\geq \epsilon)^{1-\gamma}  \big\} \,,
\end{align}
having used the fact that $\bbE[\cS_\hbeta] = 0$ in the last line.
Observe that
\begin{equation*}
	\bbE[\cS_\hbeta^2] = \sum_{k=1}^\infty
	\hbeta^{2k} \E[Y_k^2] = \E[Y_1^2] \hbeta^2 + O(\hbeta^4)
	\qquad \text{as } \hbeta \downarrow 0 \,,
\end{equation*}
hence the first two terms in \eqref{eq:trelin} give the correct asymptotic behavior,
cf. \eqref{Zfracmom}. It remains to show that the three terms in brackets
are $o(\hbeta^2)$. Note that
\begin{equation*}
	\bbE[\cT_\hbeta^2] = \sum_{k=2}^\infty \hbeta^{2k} \E[Y_k^2] = O(\hbeta^4) \,,
\end{equation*}
and moreover $\bbE[Y_1^4] \le (const.) \bbE[Y_1^2]^2 < \infty$,
by the hyper-contractivity of Wiener chaos expansions (see e.g.~\cite[Thm.~3.50]{J97}).
Writing $\cS_\hbeta = \hbeta Y_1 + \cT_\hbeta$, we obtain
\begin{equation*}
	\bbP(|\cS_\hbeta| \ge \epsilon) \le \bbP(|\hbeta Y_1| \ge \tfrac{1}{2}\epsilon)
	+ \bbP(|\cT_\hbeta| \ge \tfrac{1}{2} \epsilon)
	\le \big(\tfrac{2}{\epsilon}\big)^4 \E[Y_1^4] \hbeta^4 +
	\big(\tfrac{2}{\epsilon}\big)^2 \bbE[\cT_\hbeta^2] = O(\hbeta^4) \,.
\end{equation*}
Since $\bbE[\cS_\hbeta^2] = O(\hbeta^2)$, we can also write
\begin{equation*}
	\big| \bbE[\cS_\hbeta\ind_{\{|\cS_\hbeta| > \epsilon\}}] \big|
	\le \bbE[\cS_\hbeta^2]^{1/2} \, \bbP(|\cS_\hbeta| > \epsilon)^{1/2}
	= O(\hbeta^3) \,,
\end{equation*}
and analogously
\begin{equation*}
\begin{split}
	\bbE[\cS_\hbeta^2\ind_{\{|\cS_\hbeta| > \epsilon\}}] & \le
	2 \bbE[(\hbeta Y_1)^2\ind_{\{|\cS_\hbeta| > \epsilon\}}] +
	2 \bbE[\cT_\hbeta^2\ind_{\{|\cS_\hbeta| > \epsilon\}}] \\
	& \le 2 \hbeta^2 \bbE[Y_1^4]^{1/2} \, \bbP(|\cS_\hbeta| > \epsilon)^{1/2} +
	2 \bbE[\cT_\hbeta^2] = O(\hbeta^4) \,.
\end{split}
\end{equation*}
The terms in bracket in \eqref{eq:trelin}
are thus $O(\hbeta^3) + O(\hbeta^4) + O(\hbeta^{4(1-\gamma)})$,
which is $o(\hbeta^2)$ provided we choose $4(1-\gamma) > 2$, i.e.
$\gamma < \frac{1}{2}$.
This concludes the proof of \eqref{fWnlim2} and Theorem~\ref{T:quenchsing}.
\qed

\appendix

\section{Random Closed Subsets of $\R$}
\label{A:RCS}

In this section, we give a self-contained account
of the theoretical background needed to study random closed sets of $\R$.

\subsection{Closed subsets of $\R$}

We denote by $\cC$ the class of all closed subsets of $\R$ (including the empty set):
\begin{equation}\label{eq:C}
	\cC := \{C \subseteq \R: \ C \text{ is closed} \}.
\end{equation}
We equip the set $\cC$ with the so-called Fell-Matheron topology, built as follows.

We first compactify $\R$ by defining $\bR:= \R\cup \{\pm\infty\}$, equipped with the metric
\begin{equation}
d(x, y) = |\arctan(y) - \arctan(x)| \qquad \mbox{for all } x,y \in \bR.
\end{equation}
The Hausdorff distance of two compact non-empty subsets $K, K' \subseteq \bR$ is
defined by
\begin{equation} \label{eq:Hausdorff}
d_{\rm H}(K, K') :=
\max\Big\{ \sup_{x\in K} d(x, K') \, , \, \sup_{x'\in K'} d(x', K)
\Big\} ,
\end{equation}
where $d(a, B) :=\inf_{b\in B}d(a,b)$. (Note that
$d_H(K, K')\leq \epsilon$ if and only if for each $x\in K$ there is $x'\in K'$
with $d(x, x')\leq \epsilon$, and vice versa switching the roles of $K$ and $K'$.)

Coming back to $\cC$, one can
identify a closed subset $C \subseteq \R$ with the compact
non-empty subset $C \cup \{\pm\infty\} \subseteq \bR$. This allows to
define a metric $d_{\rm FM}$ on $\cC$:
\begin{equation} \label{eq:FM}
	d_{\rm FM}(C, C') := d_{\rm H}(C \cup \{\pm\infty\}, C' \cup \{\pm\infty\}),
	\qquad C, C' \in \cC .
\end{equation}
The topology induced by the distance $d_{\rm FM}$ on $\cC$ is called
the Fell-Matheron topology \cite[Prop. 1-4-4 and Remark on p.14]{M75}.
Since the metric space
$(\cC, d_{\rm FM})$ is \emph{compact}
(hence separable and complete), cf.~\cite[Th. 1-2-1]{M75},
it follows that $\cC$ is a \emph{Polish space}.

\begin{remark}\rm\label{rem:Mateq}
By \eqref{eq:Hausdorff}-\eqref{eq:FM}, $C_n \to C$ in $\cC$ if and only if
the following conditions hold:
\begin{itemize}
\item for every open set $G \subseteq \R$
with $G \cap C \ne \emptyset$, one has $G \cap C_n \ne \emptyset$
for large $n$;
\item for every compact set $K \subseteq \R$
with $K \cap C = \emptyset$, one has $K \cap C_n = \emptyset$
for large $n$.
\end{itemize}
We also observe that
the Fell-Matheron topology on closed subsets
can be studied for more general topological space, together with the
topology induced by the Hausdorff
metric \eqref{eq:Hausdorff} on compact non-empty subsets (called myope topology):
for more details, we refer to \cite{M75}, \cite[Appendixes B and C]{M05}
and~\cite[Appendix B]{SSS14}.
\end{remark}

\subsection{Finite-dimensional distributions}
The space $\cC$ is naturally equipped with the Borel $\sigma$-algebra $\cB(\cC)$ generated
by the open sets. By \emph{random closed subset of $\R$} we mean any $\cC$-valued
random variable $X$.
We are going to characterize the law of $X$, which is a probability measure on $\cC$,
in terms of suitable
\emph{finite-dimensional distributions}, which provide useful criteria
for convergence in distribution.

To every element $C\in \cC$ we associate two non-decreasing and right-continuous functions
$t \mapsto \tg_t(C)$ and $t \mapsto \td_t(C)$, defined for $t \in \R$ with values in $\bR$
as follows:
\begin{equation} \label{eq:gtdt}
\tg_t(C) := \sup\{x: \ x\in C, \ x \le t\}, \qquad
\td_t(C) := \inf\{x: \ x\in C, \ x > t\}
\end{equation}
(where $\sup\emptyset := -\infty$ and $\inf\emptyset := +\infty$).
Note that \emph{either function determines the set $C$}, because $t \in C$ if and only if
$\tg_t(C) = t$ if and only if $\td_{t-}(C) = t$. It is therefore natural to describe
a random closed set $X$
in terms of the random functions $t \mapsto \tg_t(X)$ and $t \mapsto \td_t(X)$.

For convenience, we state results for both
$\tg$ and $\td$, even if one could focus only on one of the two.
We start with some basic properties of the maps $\tg_t(\cdot)$ and $\td_t(\cdot)$.

\begin{lemma} \label{th:basic}
For every $t\in\R$,
consider $\tg_t(\cdot)$ and $\td_t(\cdot)$ as maps from $\cC$ to $\bR$.
\begin{ienumerate}
\item\label{it:basici} These maps are measurable with respect to
the Borel $\sigma$-algebra $\cB(\cC)$, and they generate it
as the index $t$ ranges in a dense set $\cT \subseteq \R$, i.e.\
$\cB(\cC) = \sigma((\tg_t)_{t\in\cT}) = \sigma((\td_t)_{t\in\cT})$.

\item\label{it:basicii} These maps are not continuous on $\cC$. In fact, the map
$\tg_t(\cdot)$ is continuous at $C \in \cC$
if and only if the function $\tg_\cdot(C)$
is continuous at $t$. The same holds for $\td_t(\cdot)$.
\end{ienumerate}
\end{lemma}

Given a $\cC$-valued random variable $X$,
we call \emph{$\tg$-finite-dimensional distributions ($\tg$-f.d.d.) of $X$}
the laws of the
random vectors $(\tg_{t_1}(X), \ldots, \tg_{t_k}(X))$,
for $k\in\N$ and $t_1, \ldots, t_k \in \R$.
Analogously, we call
$\td$-f.d.d. the laws of the random vectors $(\td_{t_1}(X), \ldots, \td_{t_k}(X))$.
We simply write f.d.d. to mean either $\tg$-f.d.d. or $\td$-f.d.d., or both,
when no confusion arises.

Since $X$ is determined by the functions $t \mapsto \tg_t(X)$
and $t \mapsto \td_t(X)$,
it is not surprising that the law of $X$ on $\cC$ is uniquely determined
by its f.d.d.,
and that criteria for convergence in distribution $X_n \Rightarrow X$
of $\cC$-valued random variables can be given in terms of f.d.d.. Some care is needed, however,
because the maps $\tg_t(\cdot)$ and $\td_t(\cdot)$ are not continuous on $\cC$.
For this reason, given a $\cC$-valued random variable $X$,
we denote by $\cI_\tg(X)$ the subset of those $t \in \R$ for which
the function $s \mapsto \tg_s(X)$ is continuous at
$s=t$ with probability one:
\begin{equation} \label{eq:Ig}
	\cI_\tg(X) := \big\{t \in \R: \ \P\big(\tg_{t-}(X) = \tg_{t}(X) \big) = 1 \big\}.
\end{equation}
One defines $\cI_\td(X)$ analogously.
We then have the following result.

\begin{proposition}[Characterization and convergence via f.d.d.]\label{P:cK+char}
Let $(X_n)_{n\in\N}$, $X$ be $\cC$-valued random variables.
\begin{ienumerate}
\item\label{it:convi}
The set $\cI_\tg(X)$ is cocountable, i.e. $\R\setminus\cI_\tg(X)$ is at most countable.

\item\label{it:convii}
The law of $X$ is determined by
its $\tg$-f.d.d. with indices $t_1,\ldots, t_k$ in a dense set  $\cT \subseteq \R$.

\item\label{it:conviii} Assume that $X_n \Rightarrow X$. Then
the $\tg$-f.d.d. of $X_n$ with indices in the cocountable set $\cI_\tg(X)$ converge
weakly to the $\tg$-f.d.d. of $X$:
for all $k\in\N$ and $t_1, \ldots, t_k \in \cI_\tg(X)$,
\begin{equation*}
\begin{split}
	(\tg_{t_1}(X_n), \ldots, \tg_{t_k}(X_n)) & \Rightarrow (\tg_{t_1}(X), \ldots, \tg_{t_k}(X)) .
\end{split}
\end{equation*}

\item\label{it:conviiii}
Assume that the $\tg$-f.d.d. of $X_n$ with indices in set $\cT \subseteq \R$
with full Lebesgue measure
converge weakly: for $k\in\N$, $t_1, \ldots, t_k \in \cT$ there are
measures $\mu_{t_1, \ldots, t_k}$ on $\bR^k$ such that
\begin{equation*}
	(\tg_{t_1}(X_n), \ldots, \tg_{t_k}(X_n)) \Rightarrow \mu_{t_1, \ldots, t_k}.
\end{equation*}
Then there is a $\cC$-valued random variable $X$ such that $X_n \Rightarrow X$.
In particular, the $\tg$-f.d.d. of $X$ with indices in the set $\cT \cap \cI_\tg(X)$
are given by $\mu_{t_1, \ldots, t_k}$.
\end{ienumerate}
The same conclusions hold replacing $\tg$ by $\td$.
\end{proposition}

\begin{remark}\rm\label{rem:sbo}
In Proposition~\ref{P:cK+char}~\eqref{it:conviiii} it is sufficient that
$\cT$ has uncountably many
points in every non-empty open interval $(a,b) \subseteq \R$, as the proof shows.
In fact, arguing as in \cite[Th. 7.8 in Ch. 3]{EK86}, it is even enough
that $\cT$ is dense in $\R$
(in which case the f.d.d. of $X$
must be recovered from $\mu_{t_1,\ldots,t_n}$ by a limiting procedure,
since $\cT \cap \cI_\tg(X)$ could be empty).
\end{remark}

\begin{remark}\rm\label{rem:sko}
The map $C \mapsto (\tg_t(C))_{t\in\R}$ allows one to identify
$\cC$ with a class of functions $\cD_0$ that can be explicitly described:
\begin{equation}
\begin{split}
	\cD_0:= \{ f: \R\to \R \cup \{-\infty\} : \ \, & f(t)\leq t \mbox{ and } f(t+)=f(t), \
	\forall\, t\in \R; \\
        & \text{if }  f(t)<u \text{ for some } u<t, \text{ then } f(u)=f(t) \} .
\end{split}
\end{equation}
The functions in $\cD_0$ are non-decreasing and right-continuous,
hence c\`adl\`ag,
and it turns out that the Fell-Matheron topology on $\cC$
corresponds to the \emph{Skorokhod topology} on $\cD_0$.
As a matter of fact, given the structure of $\cD_0$,
convergence $f_n \to f$ in the Skorokhod topology
is equivalent to pointwise convergence $f_n(x) \to f(x)$ at
all continuity points $x$ of $f$.

We do not prove these facts, because we do not use them directly.
However, as the reader might have noticed,
the key results of this section are translations of
analogous results for the Skorokhod topology, cf.~\cite{B99,EK86,JS03}.
\end{remark}

\subsection{Restricted finite-dimensional distributions}
\label{sec:restricted}

It turns out that,
in order to describe the f.d.d. of a $\cC$-valued random variable $X$,
say the law of $(\tg_{t_1}(X), \ldots, \tg_{t_k}(X))$,
with $-\infty < t_1 < \ldots < t_k < +\infty$, it is sufficient
to focus on the event
\begin{equation} \label{eq:restricted}
\begin{split}
	& A_{t_1,\ldots, t_k}^X :=  \{X \cap (t_{1},t_{2}] \ne \emptyset , \
	X \cap (t_{2},t_{3}] \ne \emptyset , \ \ldots\ ,
	X \cap (t_{k-1},t_{k}] \ne \emptyset\} .
\end{split}
\end{equation}
We thus define the \emph{restricted $\tg$-f.d.d.} of
$X$ as the laws of the vectors $(\tg_{t_1}(X), \ldots, \tg_{t_k}(X))$
\emph{restricted on the event $A_{t_1,\ldots, t_k}^X$}.
These are sub-probabilities, i.e.\ measures with total mass
$\P(A_{t_1,\ldots, t_k}^X) \le 1$,
and we equip them with the usual topology of weak convergence
with respect to bounded and continuous functions.
One defines analogously the restricted $\td$-f.d.d. of $X$.
We can then rephrase Proposition~\ref{P:cK+char} as follows.

\begin{proposition}[Characterization and convergence via restricted f.d.d.]\label{P:cK+char2}
Let $(X_n)_{n\in\N}$, $X$ be $\cC$-valued random variables.
\begin{ienumerate}
\item\label{it:resti}
The law of $X$
is determined by its restricted
$\tg$-f.d.d. with indices in a dense set  $\cT \subseteq \R$.

\item\label{it:restii} Assume that $X_n \Rightarrow X$. Then
the restricted $\tg$-f.d.d. of $X_n$ with indices in
the cocountable set $\cI_\tg(X) \cap \cI_\td(X)$ converge
weakly to those of $X$.

\item\label{it:restiii}
Assume that the restricted $\tg$-f.d.d. of $X_n$ with indices
in a set $\cT \subseteq \R$ with full Lebesgue measure
converge weakly to some measures $\mu^{\rm rest}_{t_1,\ldots,t_k}$ on $\bR^k$.
Then there is a $\cC$-valued random variable $X$ such that $X_n \Rightarrow X$.
In particular, the restricted $\tg$-f.d.d. of $X$ with indices in the set
$\cT \cap \cI_\tg(X) \cap \cI_\td(X)$
are given by $\mu^{\rm rest}_{t_1, \ldots, t_k}$.
\end{ienumerate}
The same conclusions hold replacing $\tg$ by $\td$.
\end{proposition}

\begin{remark}\rm
Note that $X \cap (s,t] \ne \emptyset$ if and only if $s \le \td_s(X)
\le \tg_t(X) \le t$.
Recalling the definition \eqref{eq:restr} of the set
$\cR^{(k)}_{t_0, t_1, \ldots, t_k, t_{k+1}}$, the event $A_{t_1,\ldots, t_k}^X$
in \eqref{eq:restricted} can be rewritten as
\begin{equation} \label{eq:restricted2}
	A_{t_1,\ldots, t_k}^X =
	\big\{ (\tg_{t_1}(X), \td_{t_1}(X), \ldots, \tg_{t_k}(X), \td_{t_k}(X)) \in
	\cR^{(k)}_{-\infty, t_1, \ldots, t_k, \infty} \big\} .
\end{equation}
Also note that, in case $X \subseteq [a,b]$ a.s.,
it is enough to specify the restricted f.d.d.\ with indices $t_i \in [a,b]$,
using correspondingly $\cR^{(k)}_{a, t_1, \ldots, t_k, b}$
instead of $\cR^{(k)}_{-\infty, t_1, \ldots, t_k, \infty}$ in \eqref{eq:restricted2}.
\end{remark}

\subsection{The $\alpha$-stable regenerative set}

For each $\alpha \in (0,1)$ there is a universal random closed set
$\btau$ of $[0,\infty)$,
called the \emph{$\alpha$-stable regenerative set}.
Its probability law $\bP^\alpha$ on $\cC$ is then characterized by the following
restricted f.d.d. densities: for all $k\in\N$ and $0 < t_1 < \ldots < t_k < \infty$,
setting $y_0 := 0$ and $C_\alpha:=\frac{\alpha \sin (\pi\alpha)}{\pi}$,
\begin{equation} \label{eq:regfdd}
\begin{split}
	& \frac{\bP^{\alpha} \big(\tg_{t_1}(\btau) \in \dd x_1,\,
\td_{t_1}(\btau) \in \dd y_1, \, \ldots\,,\,
\tg_{t_k}(\btau)\in \dd x_k,\, \td_{t_k}(\btau) \in \dd y_k\big)}
{\dd x_1 \, \dd y_1 \, \cdots \, \dd x_k \, \dd y_k} \\
& \qquad = \, \mathsf{f}_{t_1,\ldots,t_k}^{\alpha}(x_1,y_1,\ldots,x_k,y_k)
\, := \, \prod_{i=1}^{k}
\frac{C_\alpha}{(x_{i} - y_{i-1})^{1-\alpha}
\, (y_i-x_i)^{1+\alpha}}
\,,
\end{split}
\end{equation}
restricting $(x_1,y_1, \ldots, x_k, y_k)$ in the set $\cR^{(k)}_{0, t_1, \ldots, t_k, \infty}
\subseteq \R^{2k}$, cf.\ \eqref{eq:restricted2}-\eqref{eq:restr}.

The $\alpha$-stable regenerative set can be characterized in many ways
(e.g., as the zero level set of a Bessel process
of dimension $\delta := 2(1-\alpha) \in (0,2)$ \cite{RY99}, or
as the closure of the range of the $\alpha$-stable subordinator
\cite{FFM85}). One of the most expressive is
to view it as the \emph{universal scaling limit of discrete renewal processes}:
for any renewal process $\tau:=(\tau_n)_{n\in\N_0}$ on $\N_0$
satisfying \eqref{eq:ass},
the rescaled random set $\tau/N$, cf.\ \eqref{eq:rescN},
viewed as a $\cC$-valued random variable,
converges in distribution as $N\to\infty$ toward the
$\alpha$-stable regenerative set.

This can be shown
using the general theory of \emph{regenerative sets}, cf. \cite{FFM85},
but it is instructive to prove it directly, as an application of
Proposition~\ref{P:cK+char2} (which shows, as a by-product, that the restricted f.d.d.
\eqref{eq:regfdd} define indeed a probability law on $\cC$).
We spell this out in the \emph{conditioned case}, which
is more directly linked to our main results, but the proof for the unconditioned case
is analogous (and actually simpler).

\begin{proposition}[Characterization and universality of $\alpha$-stable regenerative set]
\label{P:univrenew}
Let $\tau:=(\tau_n)_{n\in\N_0}$ be a renewal process on $\N_0$ satisfying \eqref{eq:ass},
for some $\alpha \in (0,1)$. For fixed
$T>0$, the rescaled random set $\tau/N \cap [0,T]$, cf.\ \eqref{eq:rescN},
conditioned on $\lfloor TN\rfloor\in\tau$, converges in distribution
as $N\to\infty$ to the probability law $\bP^{\alpha;\rm c}_T$ of the
$\alpha$-stable regenerative set $\btau \cap [0,T]$
conditioned on $T\in \btau$, cf.\ \eqref{eq:bPnT}.
This law is characterized by the following
restricted f.d.d.: for all $k\in\N$ and $0<t_1<\cdots <t_k< T$,
and for $(x_1,y_1, \ldots, x_k, y_k)$ in the set $\cR^{(k)}_{0, t_1, \ldots, t_k, T} \subseteq \R^{2k}$,
cf.\ \eqref{eq:restricted2} and \eqref{eq:restr},
\begin{equation}\label{btauchar}
\begin{aligned}
& \bP^{\alpha;\rm c}_T\big(\tg_{t_1}(\btau) \in \dd x_1,\,
\td_{t_1}(\btau) \in \dd y_1, \ldots, \tg_{t_k}(\btau)\in \dd x_k,\, \td_{t_k}(\btau) \in \dd y_k\big) \\
& \qquad = \
\mathsf{f}_{T;t_1,\ldots,t_k}^{\alpha;\rm c}(x_1,y_1,\ldots,x_k,y_k) \,:=\,
\mathsf{f}_{t_1,\ldots,t_k}^{\alpha}(x_1,y_1,\ldots,x_k,y_k) \,
\frac{T^{1-\alpha}}{(T-y_k)^{1-\alpha}} \,,
\end{aligned}
\end{equation}
where $\mathsf{f}_{t_1,\ldots,t_k}^{\alpha}(\cdot)$ are
the f.d.d. densities of the $\alpha$-stable regenerative set,
cf.\ \eqref{eq:regfdd}.
\end{proposition}

\subsection{Proofs}
\label{sec:appproofs}

We now give the proofs of the results stated in the previous subsections.

\begin{proof}[Proof of Lemma~\ref{th:basic}]
We start proving part \eqref{it:basicii}. Fix $C \in \cC$ and $t \in \R$
and recall Remark~\ref{rem:Mateq}.
If the function $\tg_\cdot(C)$ is \emph{not} continuous at $t$,
i.e.\ $\tg_{t-}(C) < \tg_t(C)$, defining $C_n := C + \frac{1}{n}$
(i.e. translating $C$ to the right by $\frac{1}{n}$)
one has $C_n \to C$ as $n\to\infty$, but
$\tg_t(C_n) = \tg_{t-\frac{1}{n}}(C) + \frac{1}{n} \to \tg_{t-}(C) \ne \tg_t(C)$.
Thus $\tg_t(\cdot)$ is \emph{not} continuous at $C$.

Assume now that $\tg_\cdot(C)$ is continuous at $t$. We
set $s := \tg_t(C)$ and distinguish two cases.

\begin{itemize}

\item
If $s < t$, then $t \not\in C$ by \eqref{eq:gtdt}.
Assume for simplicity that $s > -\infty$ (the case $s = -\infty$ is analogous).
For $\epsilon > 0$ small one has $(s-\epsilon, s+\epsilon) \cap C \ne \emptyset$,
while $[s+\epsilon, t] \cap C = \emptyset$.
If $C_n \to C$, then
$(s-\epsilon, s+\epsilon) \cap C_n \ne \emptyset$ and
$[s+\epsilon, t] \cap C_n = \emptyset$
for large $n$, hence
$\tg_t(C_n) \in (s-\epsilon, s+\epsilon)$. This shows that $\tg_t(C_n) \to \tg_t(C)$, i.e.,
$\tg_t(\cdot)$ is continuous at $C$.

\item If $s = t$,
since $\tg_{t-}(C) = \tg_t(C) = t$,
for every $\epsilon > 0$ one has $(t-\epsilon, t) \cap C \ne \emptyset$.
If $C_n \to C$, then
$(t-\epsilon, t) \cap C_n \ne \emptyset$ for large $n$, hence
$\tg_t(C_n) \in (t-\epsilon, t)$. This shows that $\tg_t(C_n) \to \tg_t(C)$, that is,
$\tg_t(\cdot)$ is continuous at $C$.
\end{itemize}
Thus $\tg_t(\cdot)$ is continuous at $C \in \cC$
if and only if $\tg_\cdot (C)$ is continuous at $t$,
proving part \eqref{it:basicii}.

\smallskip

We now turn to part \eqref{it:basici}.
Defining $G_{t;m,\epsilon}(C) :=
	\frac{1}{\epsilon} \int_t^{t+\epsilon} \max\{\tg_s(C), -m\} \, \dd s$,
we can write
$\tg_t(C) = \lim_{m \to \infty} \lim_{n\to\infty} G_{t;m, 1/n}(C)$
for all $t\in\R$ and $C\in\cC$.
The measurability of $\tg_t(\cdot)$ will follow
if we show that $G_{t;m, \epsilon}(\cdot)$ is continuous, and hence measurable.
If $C_n \to C$ in $\cC$, we know that $\tg_s(C_n) \to \tg_s(C)$
at continuity points $s$ of the non-decreasing function $\tg_\cdot(C)$, hence
for Lebesgue a.e. $s\in\R$.
Since $\tg_s(C) \le s$, dominated convergence
yields $G_{t;m, \epsilon}(C_n) \to G_{t;m, \epsilon}(C)$ as $n\to\infty$,
i.e.\ the function $G_{t;m, \epsilon}(\cdot)$ is continuous on $\cC$.

Finally, setting $\cB' := \sigma((\tg_t)_{t\in\cT})$, where $\cT \subseteq \R$
is a fixed dense set, the measurability of the maps
$\tg_t(\cdot)$ yields $\cB' \subseteq \cB(\cC)$. If we
exhibit measurable maps $\psi_n: (\cC, \cB') \to (\cC, \cB(\cC))$
such that $C = \lim_{n\to\infty} \psi_n(C)$ in $\cC$, for all $C\in\cC$,
it follows that the identity map $\psi(C) := C$ is measurable from
$(\cC, \cB')$ to $(\cC, \cB(\cC))$, as the pointwise limit of $\psi_n$, hence
$\cB(\cC) \subseteq \cB'$.

Extracting a \emph{countable} dense set $\{t_i\}_{i\in\N} \subseteq \cT$,
we define $\psi_n(C) := \{\tg_{t_1}(C), \ldots, \tg_{t_n}(C)\} \cap \R$,
so that $\psi_n(C)$ is a finite subset of $\R$ and $\psi_n(C) \to C$ in $\cC$.
Since $\tg_t(\cdot)$ is a measurable map from $(\cC, \cB')$ to $\bR$ and
$(x_1, \ldots, x_n) \mapsto \{x_1, \ldots, x_k\} \cap \R$ is a continuous,
hence measurable, map from ${\bR}^k$ to $(\cC, \cB(\cC))$, it follows that
$\psi_n: (\cC, \cB') \to (\cC, \cB(\cC))$ is measurable.
\end{proof}

\begin{proof}[Proof of Proposition~\ref{P:cK+char}.]

Since the path $t \mapsto \tg_t(X)$ is increasing and right-continuous,
its discontinuity points $t$, at which $\tg_t(X) \ne \tg_{t-}(X)$,
are at most countably many, a.s..
The corresponding fact that $\P(\tg_t(X) \ne \tg_{t-}(X)) > 0$ is possible for at most
countably many $t$ follows by a classilcal argument, see
e.g.\ \cite[Section 13]{B99}. This proves part~\eqref{it:convi}.

\smallskip

The proof of part \eqref{it:convii} is an easy consequence of Lemma~\ref{th:basic}.
A generator for the Borel $\sigma$-algebra $\cB(C)$ is given by
sets of the form $\{C\in\cC: \ \tg_{t_1}(C) \in A_1, \ldots, \tg_{t_k}(C) \in A_k\}$,
for $k\in\N$, $t_1,\ldots, t_k \in \cT$ and $A_1, \ldots, A_k$ Borel
subsets of $\cB$. Note that such sets are a $\pi$-system,
i.e.\ they are closed under finite intersections.
It is then a standard result that any probability on $(\cC, \cB(C))$
--- in particular, the distribution of any $\cC$-valued random variable $X$ ---
is characterized
by its values on such sets, i.e.\ by its finite-dimensional distributions.

\smallskip

We now turn to part \eqref{it:conviii}.
If $X_n \Rightarrow X$ on $\cC$,
by Skorokhod's Representation Theorem \cite[Th. 6.7]{B99} we can couple $X_n$ and
$X$ so that a.s. $X_n \to X$ in $\cC$.
If $t_i, \ldots, t_k \in \cI_\tg(X)$, the maps $\tg_{t_i}(\cdot)$
are a.s. continuous at $X$, by Lemma~\ref{th:basic}~\eqref{it:basicii}, hence
one has the a.s. convergence
$(\tg_{t_1}(X_n), \ldots,
\tg_{t_k}(X_n)) \to (\tg_{t_1}(X), \ldots,
\tg_{t_k}(X))$, which implies weak convergence of the f.d.d..

\smallskip

We finally prove part \eqref{it:conviiii}.
Since $\cC$ is a compact Polish space, every sequence
$(X_n)_{n\in\N}$ of $\cC$-valued random variables is tight,
and hence relatively compact for the topology of convergence
in distribution \cite[Th. 2.7]{B99}.
We can then extract a subsequence $X_{n_k}$ converging in distribution to
some $\cC$-valued random variable $X$.
To show that the whole sequence $X_n$ converges to $X$,
by \cite[Th. 2.6]{B99} it is enough
to show that for any other converging subsequence $X_{n'_k} \Rightarrow X'$,
the random variables $X$ and $X'$ have the same distribution.

By assumption, the f.d.d. of $X_n$ with indices $t_1,\ldots, t_k$ in
a set $\cT \subseteq \R$ with full Lebesgue measure converge to
$\mu_{t_1,\ldots, t_k}$.
Since $X_{n_k} \Rightarrow X$,
the f.d.d. of $X$ are given by $\mu_{t_1,\ldots, t_k}$
for indices in $\cT \cap \cI_\tg(X)$, by part~\eqref{it:conviii}; analogously,
the f.d.d. of $X'$ are given by $\mu_{t_1,\ldots, t_k}$
for indices in $\cT \cap \cI_\tg(X')$.
Thus $X$ and $X'$ have the same f.d.d. with indices in
$\cT \cap \cI_\tg(X) \cap \cI_\tg(X')$. Since this set is dense,
$X$ and $X'$ have the same distribution,
by Proposition~\ref{P:cK+char}~\eqref{it:convii}.
\end{proof}

\begin{proof}[Proof of Proposition~\ref{P:cK+char2}]

We start proving part \eqref{it:resti}.
Fix a $\cC$-valued random variable $X$ and
let $\mu_{t_1, \ldots, t_k}(\dd x_1, \ldots, \dd x_k)$ be
the $\tg$-f.d.d. of $X$, i.e. the law on $\bR^k$ of $(\tg_{t_1}(X), \ldots, \tg_{t_k}(X))$.
Analogously, let $\mu^{\rm rest}_{t_1, \ldots, t_k}(\dd x_1, \ldots, \dd x_k)$
be the restricted $\tg$-f.d.d. of $X$, cf. \eqref{eq:restricted}-\eqref{eq:restricted2}.
Since $X \cap (s,t] \ne \emptyset$ if and only if $\tg_t(X) \in (s,t]$, the restricted f.d.d.
$\mu^{\rm rest}_{t_1, \ldots, t_k}$ is just
the f.d.d. $\mu_{t_1, \ldots, t_k}$ restricted on the subset
$[-\infty,t_1] \times (t_1,t_2] \times \ldots \times (t_{k-1},t_k]$:
\begin{equation}\label{eq:restrmu}
	\mu^{\rm rest}_{t_1, \ldots, t_k}(\dd x_1, \ldots, \dd x_k)
	= \mu_{t_1, \ldots, t_k}(\dd x_1, \ldots, \dd x_k)
	\prod_{i=2}^k \ind_{\{x_i \in (t_{i-1},t_i]\}} .
\end{equation}

To prove that the f.d.d. can be recovered from the restricted f.d.d., we show that
$\mu_{t_1, \ldots, t_k}(\cdot)$ can be written as a mixture of
$\mu^{\rm rest}_{I}(\cdot)$, as $I$ ranges in the subsets of
$\{t_1, \ldots, t_k\}$. For $k=1$ there is nothing to prove, since $\mu_{t_1} = \mu^{\rm rest}_{t_1}$,
so we assume $k\ge 2$ henceforth.
We associate to $X$ the (possibly empty) subset
$\mathscr B(X) \subseteq \{1, \ldots, k-1\}$ defined by
\begin{equation*}
	\mathscr B(X) \,:=\, \big\{ j \in \{1, \ldots, k-1\}: \ X \cap (t_{j}, t_{j+1}] \ne \emptyset
	\big\} \,.
\end{equation*}
Performing a decomposition over the possible values of $\mathscr B(X)$, we can write
\begin{equation} \label{eq:muQdecomp}
	\mu_{t_1, \ldots, t_k}(\dd x_1, \ldots, \dd x_k)
	\,= \!\!\! \sum_{B \subseteq \{2,\ldots, k\}} \!\!\!
	\P(\tg_{t_1}(X) \in \dd x_1,\, \ldots,\, \tg_{t_k}(X) \in \dd x_k , \, \mathscr B(X) = B) \,.
\end{equation}
It remains to express each term in the right hand side in terms of restricted f.d.d..

We first consider the case $B = \emptyset$.
On the event $\{\mathscr B(X) = \emptyset\}=\{X \cap (t_1, t_k] = \emptyset\}$ we have
$\tg_{t_1}(X) = \cdots = \tg_{t_k}(X) \le t_1$. Recalling that
$\mu_{t} = \mu^{\rm rest}_{t}$, we can write
\begin{equation*}
\begin{split}
	& \P(\tg_{t_1}(X) \in \dd x_1,\, \ldots,\, \tg_{t_k}(X) \in \dd x_k , \, \mathscr B(X) = \emptyset)\\
	& \quad =\ \P(\tg_{t_k}(X) \in \dd x_k) \ \ind_{\{x_k \le t_1\}} \,
	\prod_{i=1}^{k-1}
	\delta_{x_k}(\dd x_i)
	\ = \ \mu^{\rm rest}_{t_k}(\dd x_k) \ \ind_{\{x_k \le t_1\}} \,
	\prod_{i=1}^{k-1}
	\delta_{x_k}(\dd x_i) ,
\end{split}
\end{equation*}
and this expression depends only on the restricted f.d.d..

If $B \ne \emptyset$, we can write $B = \{j_1, \ldots, j_\ell\}$ with $1 \le \ell \le k-1$
and $1 \le j_1 < \ldots < j_\ell \le k-1$.
Let us also set $j_0 := 0$ and $j_{\ell + 1} := k$.
On the event $\{\mathscr B = B\}$,
we have $\tg_{t_{j_{n-1}+1}}(X) = \tg_{t_{j_{n-1}+2}}(X) = \ldots = \tg_{t_{j_n}}(X)
\in (t_{j_{n-1}}, t_{j_{n-1}+1}]$, for every $n=1, \ldots, \ell + 1$, therefore
\begin{equation}\label{eq:aassa}
\begin{split}
	& \P(\tg_{t_1}(X) \in \dd x_1,\, \ldots,\, \tg_{t_k}(X) \in \dd x_k ,
	\, \mathscr B(X) = \{j_1, \ldots, j_\ell\}) \\
	& = \P\Bigg( \bigcap_{n=1}^{\ell+1} \big\{ \tg_{t_{j_{n}}}(X) \in \dd x_{j_{n}} \big\} \Bigg)
	\Bigg\{\prod_{n=1}^{\ell+1}
	\ind_{\{x_{j_{n}} \in (t_{j_{n-1}}, t_{j_{n-1}+1}]\}}
	\Bigg(\prod_{m=j_{n-1}+1}^{j_{n} - 1}
	\!\delta_{x_{j_{n}}}( \dd x_{m}) \Bigg) \Bigg\} ,
\end{split}
\end{equation}
where we set $(t_0, \,\cdot\,] := [-\infty, \,\cdot\,]$ and
the product over $m$ equals one when $j_{n} - j_{n-1} = 1$.
The first term in the right hand side of \eqref{eq:aassa} is the f.d.d.
$\mu_{t_{j_1}, \ldots, t_{j_{\ell+1}}}(\dd x_{j_1}, \ldots, \dd x_{j_{\ell+1}})$.
However, by \eqref{eq:restrmu}, \emph{this coincides with the restricted f.d.d.}
$\mu^{\rm rest}_{t_{j_1}, \ldots, t_{j_{\ell+1}}}(\dd x_{j_1}, \ldots, \dd x_{j_{\ell+1}})$,
because each variable $x_{j_n}$ is restricted on $(t_{j_{n-1}}, t_{j_{n-1}+1}]
\subseteq (t_{j_{n-1}}, t_{j_{n}}]$. Part \eqref{it:resti} is thus proved.

\smallskip

For part \eqref{it:restii},
we proceed as in the proof of Proposition~\ref{P:cK+char}~\eqref{it:conviii}.
If $X_n \Rightarrow X$ on $\cC$, we couple $X_n$ and
$X$ so that a.s. $X_n \to X$ in $\cC$,
by Skorokhod's Representation Theorem.
On the event that $\tg_\cdot(X)$ is continuous at $t$,
if $X \cap (s,t] \ne \emptyset$
then also $X \cap (s,t) \ne \emptyset$, which implies
$X_n \cap (s,t) \ne \emptyset$ for large $n$ (cf.\ Remark~\ref{rem:Mateq}).
Analogously, on the event that $\td_\cdot(X)$ is continuous at $s$,
if $X \cap (s,t] = \emptyset$
then also $X \cap [s,t] = \emptyset$, which implies
$X_n \cap [s,t] = \emptyset$ for large $n$.
Therefore, if $t_1, \ldots, t_k \in \cI_\tg(X) \cap \cI_\td(X)$
and $f:\bR^k \to \R$ is bounded and continuous,
\begin{equation*}
\begin{split}
	& f\big( \tg_{t_1}(X_n), \ldots, \tg_{t_k}(X_n) \big)
	\ind_{\{X_n \cap (t_{i-1},t_{i}] \ne \emptyset,
	\ \forall i = 2, \ldots, k\} } \\
	& \qquad \qquad \qquad \qquad \xrightarrow[\ n\to\infty \ ]{\text{a.s.}}
	f\big( \tg_{t_1}(X), \ldots, \tg_{t_k}(X) \big)
	\ind_{\{X \cap (t_{i-1},t_{i}] \ne \emptyset,
	\ \forall i = 2, \ldots, k\} } .
\end{split}
\end{equation*}
Taking expectations of both sides, dominated convergence
shows that the restricted f.d.d. of $X_n$
with indices in $\cI_\tg(X) \cap \cI_\td(X)$
converge weakly toward the restricted f.d.d. of $X$.

\smallskip

Finally, the proof of part \eqref{it:restiii}
is analogous to that of
Proposition~\ref{P:cK+char}~\eqref{it:conviiii}.
Any sequence $X_n$
of $\cC$-valued random variable is tight, hence it suffices to show that
if $X_{n_k}$, $X_{n'_k}$ are subsequences converging in distribtion
to $X$, $X'$ respectively, then $X$ and $X'$ have the same law.
Since the restricted $\tg$-f.d.d. of $X_n$ with indices in $\cT$ converge,
$X$ and $X'$ have the same restricted
$\tg$-f.d.d. with indexes in the dense set $\cT \cap \cI_\tg(X) \cap
\cI_\td(X) \cap \cI_\tg(X') \cap \cI_\td(X')$, by part \eqref{it:restii}. It follows by part
\eqref{it:resti} that $X$ and $X'$ have the same law.
\end{proof}

\begin{proof}[Proof of Proposition~\ref{P:univrenew}]
By Proposition~\ref{P:cK+char2}~\eqref{it:restiii}, it is enough
to prove the convergence of restricted f.d.d.: for all $k\in\N$, $0 < t_1 < \ldots < t_k < T$
and for every bounded and continuous function
$F: \R^{2k} \to \R$, recalling \eqref{eq:restricted}-\eqref{eq:restricted2}
and \eqref{eq:PN}, we show that
the integral
\begin{equation}\label{eq:topro}
\begin{split}
	I_N := \E_N^{\rm c} \Big[ F\big(\tg_{t_1}(\tau/N), &
	\, \td_{t_1}(\tau/N), \ldots, \tg_{t_k}(\tau/N), \td_{t_k}(\tau/N)\big)
	\ind_{A_{t_1,\ldots, t_k}^{\tau/N}} \Big]
\end{split}
\end{equation}
converges as $N\to\infty$ to the integral of $F$
with respect to the density $\mathsf{f}_{T;t_1,\ldots,t_k}^{\alpha;\rm c}$ in \eqref{btauchar}, i.e.
\begin{equation} \label{eq:intI}
\begin{split}
	I := \idotsint\limits_{0<x_1<t_1<y_1<x_2<t_2\atop <\cdots <x_k<t_k<y_k<T}
	& F(x_1,y_1,\ldots, x_k, y_k) \\
	& \times \Bigg[ \prod_{i=1}^{k}
	\frac{C_\alpha}{(x_{i} - y_{i-1})^{1-\alpha}
	(y_i-x_i)^{1+\alpha}} \Bigg] \frac{T^{1-\alpha}}{(T-y_k)^{1-\alpha}} \, \dd x \, \dd y \,,
\end{split}
\end{equation}
where $y_0 := 0$ and
$\dd x \, \dd y$ is a shorthand for $\dd x_1 \dd y_1 \cdots \dd x_k \dd y_k$.

Recall that $u(i) := \P(i\in \tau)$ and $K(i):=\P(\tau_1=i)$. A renewal decomposition yields
\begin{equation} \label{eq:Inbasic}
\begin{split}
	I_N = \frac{1}{N^{2k}} \sum_{0 \le a_1 \le Nt_1} \, \sum_{N t_1 < b_1 \le a_2 \le Nt_2}
	\cdots & \sum_{N t_{k-1} < b_{k-1} \le a_k \le Nt_k} \, \sum_{N t_k < b_k \le NT}
	F\Big(\frac{a_1}{N}, \frac{b_1}{N}, \ldots, \frac{a_k}{N}, \frac{b_k}{N}\Big) \\
	& \times \Bigg[ \prod_{i=1}^{k}
	N^2 \, u(a_i - b_{i-1}) K(b_i - a_i) \Bigg] \frac{u(\lfloor NT\rfloor-b_k)}{u(\lfloor NT\rfloor)} ,
\end{split}
\end{equation}
where we write the factor $\frac{1}{N^{2k}}$ (which cancels with the product of the $N^2$
inside the square brackets) to make $I_N$ appear as a Riemann sum.
Setting $x_i = a_i/N$, $y_i = b_i/N$ for $1\leq i\leq k$,
the summand converges pointwise to the integrand in \eqref{eq:intI}, since
by \eqref{eq:ass} and \eqref{un}
\begin{equation*}
	\forall\, z > 0: \qquad
	K(\lfloor Nz \rfloor)\sim \frac{L(Nz)}{(Nz) ^{1+\alpha}}, \qquad
	u(\lfloor Nz\rfloor)\sim \frac{C_\alpha}{L(Nz) (Nz)^{1-\alpha}}, \qquad
	\text{as } N\to\infty .
\end{equation*}

To conclude that $I_N$ converges to the integral $I$ in \eqref{eq:intI},
we provide a suitable domination.
By Potter bounds \cite[Theorem 1.5.6]{BGT87},
for every $\epsilon > 0$ there is a constant $D_\epsilon$ such that
$L(m)/L(\ell) \le D_\epsilon \max\{ (m/\ell)^{\epsilon}, (\ell/m)^{\epsilon}\}$
for all $\ell,m\in\N$. Since $(y_{i}-x_{i}) \le T$, $(x_{i} - y_{i-1}) \le T$ and
$\max\{\alpha,\beta\} \le \alpha\beta$ for
$\alpha,\beta \ge 1$, we can write
\begin{align*}
	\frac{L(b_i-a_i)}{L(a_i-b_{i-1})} & \le
	D_\epsilon \max \bigg\{ \frac{ (y_{i}-x_{i})^{\epsilon}}{(x_{i} - y_{i-1})^{\epsilon}},
	\frac{(x_{i} - y_{i-1})^{\epsilon}}{ (y_{i}-x_{i})^{\epsilon}} \bigg\}
	\le \frac{D_\epsilon \, T^{2\epsilon}}{(x_{i} - y_{i-1})^{\epsilon}
	(y_{i}-x_{i})^{\epsilon}} ,\\
	\frac{L(\lfloor NT\rfloor)}{L(\lfloor NT\rfloor-b_k)} & \le
	D_\epsilon \frac{ T^{\epsilon}}{(T - y_{k})^{\epsilon}} .
\end{align*}
It follows that, for every $\epsilon > 0$, the summand in $I_N$ is bounded
uniformly in $N$ by
\begin{equation} \label{eq:domination}
	C(T,\epsilon)
	\Bigg[ \prod_{i=1}^{k}
	\frac{1}{(x_{i} - y_{i-1})^{1-\alpha+\epsilon}
	(y_i-x_i)^{1+\alpha+\epsilon}} \Bigg] \frac{1}{(T-y_k)^{1-\alpha+\epsilon}} .
\end{equation}

It remains to show that, if we choose $\epsilon > 0$ sufficiently small,
this function has finite integral over the domain of integration in \eqref{eq:intI}.
Let us set $\eta := \min_{i=1,\ldots,k+1} (t_{i} - t_{i-1})$ and
\begin{equation*}
	\delta_i := x_i - y_{i-1} \quad \text{for } i=1,\ldots,k+1, \qquad
	\delta'_i := y_i - x_i \quad \text{for } i=1,\ldots,k,
\end{equation*}
where $t_0 = y_0 := 0$ and $x_{k+1} = t_{k+1} := T$.
Each of the quantities $\delta_i$, $\delta'_i$ can be smaller or
larger than $\eta/3$, and we split the integral of \eqref{eq:domination} accordingly,
as a sum of $2^{2k+1}$ terms.
Note that if $\delta_i < \eta/3$, either
$\delta'_{i-1}$ or $\delta'_i$ must exceed $\eta/3$ (because
$x_{i-1} < t_{i-1} < t_i < y_i$ and $t_{i} - t_{i-1} \ge \eta$).
Whenever any $\delta_i$ or $\delta'_i$ exceeds $\eta/3$, we replace
them by $\eta/3$, getting an upper bound. This
yields a factorization into a product of just four kind of basic integrals, i.e.
\begin{equation*}
	\idotsint\limits_{ y_{i-1} < x_i < t_i < y_i < x_{i+1}\atop
	x_i- y_{i-1} < \frac{\eta}{3},\ y_i-x_i < \frac{\eta}{3},\ x_{i+1}-y_i < \frac{\eta}{3}}
	\!\!\!\!\!\!\!\!\!\!\!\!\!\!
	\frac{\dd y_{i-1}}{(x_i - y_{i-1})^{1-\alpha+\epsilon}}
	\frac{\dd x_{i} \, \dd y_i}{(y_i - x_{i})^{1+\alpha+\epsilon}}
	\frac{\dd x_{i+1}}{(x_{i+1} - y_{i})^{1-\alpha+\epsilon}},
\end{equation*}
and the analogous ones without the integration over $y_{i-1}$ and/or $x_{i+1}$.
The finiteness of such integrals is easily checked if $\epsilon > 0$ is small (so that
$1-\alpha+\epsilon < 1$ and $1+\alpha+\epsilon < 2$).
\end{proof}

\section{Renewal estimates}\label{app:renew_est}
\label{sec:dimo}
In this section we show that assumption \eqref{eq:renass0+} holds true for renewal processes
satisfying \eqref{eq:ass}, under either of the following assumptions:
\begin{itemize}
\item $\alpha > 1$ (in particular, the renewal process has finite mean);
\item the renewal process is generated by a Bessel-like Markov chain, in the spirit of \cite{A11}.
\end{itemize}
Actually, it can be shown that condition \eqref{eq:renass0+} is satisfied in much greater generality,
e.g.\ when one has an analogue of \eqref{eq:renass0+}, but
with $u(\cdot)$ replaced by the renewal kernel $K(\cdot)$.
This renewal-theoretic framework falls out of the scope of the present work and
will be treated elsewhere.

Let us, now, verify \eqref{eq:renass0+} in the two cases, mentioned above.
Note that the precise value of $\epsilon > 0$ therein is immaterial:
if $\epsilon'n \le \ell \le \epsilon n$ for
$0 < \epsilon' < \epsilon$, relation \eqref{eq:renass0+} is always satisfied
(with $\delta = 1$), as it follows by \eqref{un}. We then set $\epsilon = \frac{1}{4}$
for simplicity and rewrite \eqref{eq:renass0+} as
\begin{equation}\label{eq:renass00+}
	\exists C, n_0 \in (0,\infty), \ \delta \in (0,1]: \qquad
	\bigg| \frac{u(n+\ell)}{u(n)} - 1 \bigg| \le C \bigg(\frac{\ell}{n}\bigg)^\delta
	\quad \ \forall n \ge n_0, \ 0 \le \ell \le \frac{n}{4} \,.
\end{equation}

\begin{lemma}\label{alhpa>1}
Let $\tau$ be a non-terminating renewal process
satisfying \eqref{eq:ass}, with $\alpha>1$. Then \eqref{eq:renass00+} holds true.
\end{lemma}

\begin{proof}
In the case $\alpha>1$, by \eqref{eq:ass}
we have $\E[\tau_1^\eta]<\infty$ for some $\eta > 1$.
We can then apply the following result of Rogozin~\cite{R73}:
\begin{equation*}
	u(n) = \p(n\in \tau)
	= \frac{1}{\e[\tau_1]} + \frac{1}{\e[\tau_1]^2} \sum_{k=n+1}^\infty \p(\tau_1> k) + R_n,
\end{equation*}
where $R_n=o(n^{-2(\eta-1)})$
if $\eta \in (1,2)$, and $R_n=o(n^{-\eta})$ if $\eta\geq 2$,
as $n\to\infty$. Note that,
for any $\alpha' \in (1, \alpha)$, by \eqref{eq:ass}
we can choose $C$ such that for all $k\in\N_0$
$$
\P(\tau_1 > k) = \sum_{n=k+1}^\infty \frac{L(n)}{n^{1+\alpha}}
\leq \sum_{n=k+1}^\infty \frac{C}{n^{1+\alpha'}} \leq \frac{C'}{k^{\alpha'}}.
$$
It follows that for $n\in\N$ large enough and $0 \le \ell \le \frac{n}{4}$
$$
\begin{aligned}
\bigg| \frac{u(n+\ell)}{u(n)} - 1 \bigg| & = \bigg| \frac{u(n) - u(n+\ell)}{u(n)} \bigg|
\leq C \Big(\sum_{k=n+1}^{n+\ell} \P(\tau_1> k) + \frac{1}{n^{2(\eta-1)}}
+ \frac{1}{(n+\ell)^{2(\eta-1)}}\Big) \\
& \leq C' \Big( \frac{\ell}{n^{\alpha'}} + \frac{1}{n^{2(\eta-1)}}\Big) \leq
C''\Big(\frac{\ell}{n}\Big)^{1 \wedge (2\eta-2)}.
\end{aligned}
$$
This establishes \eqref{eq:renass00+} with
$\delta=1 \wedge (2\eta-2)$.
\end{proof}

\begin{lemma}\label{lem:Bessel1}
Let $\tau$ be a non-terminating renewal process
satisfying \eqref{eq:ass}, with $\alpha\in(0,1)$, such that $2\tau_1$ has the same distribution
as the first return to zero of a nearest-neighbor Markov chain on $\N_0$
with $\pm1$ increments (cf.\ Remark~\ref{R:Ken}). Then
\eqref{eq:renass00+} holds true.
\end{lemma}

\begin{proof}
Let $X$ and $Y$ be two copies of such a Markov chain,
starting at the origin at times $-2\ell$ and  $0$ respectively, so that
$$
u(n+\ell) = \p(X_{2n}=0) \qquad \mbox{and} \qquad u(n) = \p(Y_{2n}=0).
$$
(Although $X_n$ is defined for $n \ge -2\ell$, we only look at it for $n \ge 0$.)
We can couple $X$ and $Y$ such that they are independent until they meet, at which time they
coalesce. Since $X$ and $Y$ are nearest neighbor walks, $X_{2n}=0$ implies $Y_{2n}=0$, and
\begin{equation*}
\begin{split}
0  & \leq  \p(Y_{2n}=0) - \p(X_{2n}=0) = u(n) - u(n+\ell)  \\
& = \p(Y_{2n}=0\neq X_{2n}) = \p(Y_{2n}=0, X_t\neq 0 \ \forall\ t\in [0, 2n]) \\
& \leq \p(Y_{2n}=0) \p(X_t\neq 0 \ \forall\ t\in [0, 2n])  \ \leq\
u(n) \sum_{k=0}^{\ell-1} u(k) \p(\tau_1>n+\ell-k) \\
& \leq u(n) \p(\tau_1>n) \sum_{k=0}^{\ell-1} u(k),
\end{split}
\end{equation*}
where by the properties of regularly varying functions, see~\cite[Prop.~1.5.8 \& 1.5.10]{BGT87},
\begin{equation}\label{unsumlim}
\begin{aligned}
\P(\tau_1 >n) & = \sum_{k=n+1}^\infty \frac{L(k)}{k^{1+\alpha}} \sim
\frac{L(n)}{\alpha n^\alpha} \qquad \mbox{as } n\to\infty,  \\
\sum_{k=0}^{\ell-1} u(k) & = \sum_{k=0}^{\ell-1}
\frac{C_\alpha(1+o(1))}{L(k) k^{1-\alpha}} \sim \frac{C_\alpha \ell^\alpha}{\alpha L(\ell)}
\quad \mbox{as } \ell \to\infty.
\end{aligned}
\end{equation}
Observe that, for every $\epsilon' > 0$ there exists $n_0 < \infty$ such that
$L(n)/L(\ell) \le (n/\ell)^{\epsilon'}$ for $n \ge n_0$ and $\ell \le \frac{n}{4}$,
by Potter bounds \cite[Theorem 1.5.6]{BGT87}, hence
\begin{equation*}
	\frac{u(n) - u(n+\ell)}{u(n)} \le C \, \frac{L(n)}{L(\ell)} \, \bigg(\frac{\ell}{n}\bigg)^\alpha
	\le C \, \bigg(\frac{\ell}{n}\bigg)^{\alpha-\epsilon'},
\end{equation*}
and therefore \eqref{eq:renass00+} holds true for any $\delta<\alpha$.
\end{proof}

\section{An Integral Estimate}\label{app:int_est}

\begin{lemma}\label{L:integralbd} Let $\chi \in [0,1)$. Then there exist $C_1, C_2>0$ such
that for all $k\in\bbN$,
\begin{equation}\label{eq:intgbd1}
\idotsint\limits_{0<t_1<\cdots<t_k<1} \frac{\dd t_1\cdots \dd t_k}{t_1^{\chi}\cdots
(t_k-t_{k-1})^{\chi}(1-t_k)^{\chi}} \leq C_1 e^{-C_2k\log k}.
\end{equation}
Furthermore, for any $v\in (0,1)$ and $k_1, k_2\in\N_0:=\N\cup\{0\}$ with $k:=k_1+k_2\geq 1$,
\begin{equation}\label{eq:intgbd2}
\idotsint\limits_{0<t_1<\cdots t_{k_1}<v \atop v<t_{k_1+1}<\cdots<t_{k}<1} \frac{\dd t_1\cdots
\dd t_{k}}{t_1^{\chi}(t_2-t_1)^{\chi}\cdots (1-t_{k})^{\chi}}
\leq  C_1 e^{-C_2 k\log k}\ v^{(1-\chi)k_1} (1-v)^{(1-\chi)k_2}.
\end{equation}
\end{lemma}

\begin{proof}
By the definition of the Dirichlet distribution
$$
\idotsint\limits_{0<t_1<\cdots<t_k<1} \frac{\dd t_1\cdots \dd t_k}{t_1^{\chi}\cdots
(t_k-t_{k-1})^{\chi}(1-t_k)^{\chi}} =\frac{(\Gamma(1-\chi))^{k+1}}{\Gamma\big((k+1)(1-\chi)\big)}.
$$
The bound \eqref{eq:intgbd1} then follows from the properties of the gamma function $\Gamma(\cdot)$.

If $k_1=0$, then \eqref{eq:intgbd2} can be obtained from \eqref{eq:intgbd1} noting that
$1/t_1^{\chi} \le (1-v)^\chi/(t_1-v)^{\chi}$ and performing a change of variable. The
same applies to the case $k_2=0$. If $k_1, k_2\geq 1$, then denoting the integral in \eqref{eq:intgbd1}
by $A_k$, we have
$$
\begin{aligned}
& \idotsint\limits_{0<t_1<\cdots t_{k_1}<v \atop v<t_{k_1+1}<\cdots<t_{k_1+k_2}<1}
\frac{\dd t_1\cdots \dd t_{k_1+k_2}}{t_1^{\chi}(t_2-t_1)^{\chi}\cdots (1-t_{k_1+k_2})^{\chi}} \\
= & \iint\limits_{0<t_{k_1}<v \atop v<t_{k_1+1}<1}  \frac{\dd t_{k_1}
\dd t_{k_1+1}}{(t_{k_1+1}-t_{k_1})^{\chi}}\ A_{k_1-1} t_{k_1}^{k_1(1-\chi)-1} A_{k_2-1}
(1-t_{k_1+1})^{k_2(1-\chi)-1} \\
\leq &\ \ C_1^2 e^{-C_2(k_1-1)\log
(k_1-1)}e^{-C_2(k_2-1)\log(k_2-1)} \!\!\!\!\!\!\!\!\!\!\!\!\!\! \iint\limits_{0<t_{k_1}<v<t_{k_1+1}<1}
\!\!\!\!\!\!\!\!\!\!\!\! \frac{t_{k_1}^{k_1(1-\chi)-1}(1-t_{k_1+1})^{k_2(1-\chi)-1}}
{(t_{k_1+1}-t_{k_1})^{\chi}} \dd t_{k_1} \dd t_{k_1+1} \\
\leq &\ \ C_3 e^{-C_4(k_1+k_2)\log(k_1+k_2)}\ v^{(1-\chi)k_1} (1-v)^{(1-\chi)k_2},
\end{aligned}
$$
where the last inequality is obtained by first noting that $\max\{k_1, k_2\}\geq (k_1+k_2)/2$,
and then replacing $t_{k_1+1}-t_{k_1}$ with $t_{k_1+1}-v$ (resp.\ $v-t_{k_1}$) if $v<1/2$
(resp.\ $v\geq 1/2$). In this way the integral factorizes and one obtains
\eqref{eq:intgbd2} after adjusting the values of $C_1$ and $C_2$.
\end{proof}

\bigskip

\noindent
{\bf Acknowledgements.} 
F.C. is partially supported by ERC Advanced Grant 267356 VARIS.
R.S. is supported by AcRF Tier 1 grant R-146-000-148-112.
N.Z. is supported by EPSRC grant EP/L012154/1 and Marie 
Curie International Reintegration Grant IRG-246809.


\end{document}